\documentclass[12pt]{amsart}
\usepackage{amscd,amssymb,latexsym}
\usepackage{fullpage}
\newcommand{\Mdef}[2]{\newcommand{#1}{\relax \ifmmode #2 \else $#2$\fi}}

\newcommand{\codim}{\mathrm{codim}}

\newcommand{\cok}{\mathrm{cok}}

\newcommand{\supp}{\mathrm{supp}}

\newcommand{\im}{\mathrm{im}}

\newcommand{\sm }{\wedge}

\newcommand{\tensor}{\otimes}

\newcommand{\Hom}{\mathrm{Hom}}

\newcommand{\Ext}{\mathrm{Ext}}
\Mdef{\bhom}{\mathbf{\hat{H}om}}

\Mdef{\Mod}{\mathrm{mod}}

\newcommand{\st}{\; | \;}

\newtheorem{thm}{Theorem}[section]
\newtheorem{lemma}[thm]{Lemma}
\newtheorem{prop}[thm]{Proposition}
\newtheorem{cor}[thm]{Corollary}

\theoremstyle{definition}

\newtheorem{defn}[thm]{Definition}

\newtheorem{example}[thm]{Example}

\newtheorem{remark}[thm]{Remark}

\newcommand{\qqed}{\qed \\[1ex]}
\renewenvironment{proof}[1][\hspace*{-.8ex}]{\noindent {\bf Proof #1:\;}}{\qqed}

\Mdef{\PH} {\Phi^H}
\Mdef{\PK} {\Phi^K}
\Mdef{\PL} {\Phi^L}
\Mdef{\PT} {\Phi^{\T}}

\Mdef{\ef}{E{\cF}_+}
\Mdef{\etf}{\widetilde{E}{\cF}}
\Mdef{\eg}{E{G}_+}
\Mdef{\etg}{\tilde{E}{G}}

\newcommand{\piAt}{\pi^{\cAt}}

\Mdef{\infl}{\mathrm{inf}}
\Mdef{\defl}{\mathrm{def}}
\Mdef{\res}{\mathrm{res}}
\Mdef{\ind}{\mathrm{ind}}
\Mdef{\coind}{\mathrm{coind}}

\Mdef{\univ}{\mathcal{U}}

\Mdef{\Fp}{\mathbb{F}_p}
\Mdef{\Zpinfty}{\Z /p^{\infty}}
\Mdef{\Zpadic}{\Z_p^{\wedge}}

\newcommand{\dichotomy}[2]{\left\{ \begin{array}{ll}#1\\#2 \end{array}\right.}
\newcommand{\trichotomy}[3]{\left\{ \begin{array}{ll}#1\\#2\\#3 \end{array}\right.}

\newcommand{\lra}{\longrightarrow}
\newcommand{\lla}{\longleftarrow}

\newcommand{\lr}[1]{\langle #1 \rangle}

\newcommand{\crings}{\mathbf{Rings}}

\newcommand{\Gspectra}{\mbox{$G$-{\bf spectra}}}

\newcommand{\spec}{\mathrm{Spec}}

\Mdef{\we}{\mathbf{we}}
\Mdef{\fib}{\mathbf{fib}}
\Mdef{\cof}{\mathbf{cof}}
\Mdef{\BI}{\mathcal{BI}}

\newcommand{\ann}{\mathrm{ann}}
\newcommand{\fibre}{\mathrm{fibre}}
\newcommand{\cofibre}{\mathrm{cofibre}}

\newcommand{\ilim}{\mathop{ \mathop{\mathrm{lim}} \limits_\leftarrow} \nolimits}
\newcommand{\colim}{\mathop{  \mathop{\mathrm {lim}} \limits_\rightarrow} \nolimits}

\Mdef{\B}{\mathbb{B}}
\Mdef{\C}{\mathbb{C}}
\Mdef{\D}{\mathbb{D}}
\Mdef{\E}{\mathbb{E}}
\Mdef{\T}{\mathbb{T}}
\Mdef{\F}{\mathbb{F}}
\Mdef{\G}{\mathbb{G}}
\Mdef{\I}{\mathbb{I}}
\Mdef{\N}{\mathbb{N}}
\Mdef{\Q}{\mathbb{Q}}
\Mdef{\R}{\mathbb{R}}
\Mdef{\bbS}{\mathbb{S}}
\Mdef{\Z}{\mathbb{Z}}

\Mdef{\bA}{\mathbb{A}}
\Mdef{\bB}{\mathbb{B}}
\Mdef{\bC}{\mathbb{C}}
\Mdef{\bD}{\mathbb{D}}
\Mdef{\bE}{\mathbb{E}}
\Mdef{\bF}{\mathbb{F}}
\Mdef{\bG}{\mathbb{G}}
\Mdef{\bH}{\mathbb{H}}
\Mdef{\bI}{\mathbb{I}}
\Mdef{\bJ}{\mathbb{J}}
\Mdef{\bK}{\mathbb{K}}
\Mdef{\bL}{\mathbb{L}}
\Mdef{\bM}{\mathbb{M}}
\Mdef{\bN}{\mathbb{N}}
\Mdef{\bO}{\mathbb{O}}
\Mdef{\bP}{\mathbb{P}}
\Mdef{\bQ}{\mathbb{Q}}
\Mdef{\bR}{\mathbb{R}}
\Mdef{\bS}{\mathbb{S}}
\Mdef{\bT}{\mathbb{T}}
\Mdef{\bU}{\mathbb{U}}
\Mdef{\bV}{\mathbb{V}}
\Mdef{\bW}{\mathbb{W}}
\Mdef{\bX}{\mathbb{X}}
\Mdef{\bY}{\mathbb{Y}}
\Mdef{\bZ}{\mathbb{Z}}

\Mdef{\cA}{\mathcal{A}}
\Mdef{\cB}{\mathcal{B}}
\Mdef{\cC}{\mathcal{C}}
\Mdef{\mcD}{\mathcal{D}} 
\Mdef{\cE}{\mathcal{E}}
\Mdef{\cF}{\mathcal{F}}
\Mdef{\cG}{\mathcal{G}}
\Mdef{\mcH}{\mathcal{H}} 
\Mdef{\cI}{\mathcal{I}}
\Mdef{\cJ}{\mathcal{J}}
\Mdef{\cK}{\mathcal{K}}
\Mdef{\mcL}{\mathcal{L}}

\Mdef{\cM}{\mathcal{M}}
\Mdef{\cN}{\mathcal{N}}
\Mdef{\cO}{\mathcal{O}}
\Mdef{\cP}{\mathcal{P}}
\Mdef{\cQ}{\mathcal{Q}}
\Mdef{\mcR}{\mathcal{R}}
\Mdef{\cS}{\mathcal{S}}
\Mdef{\cT}{\mathcal{T}}
\Mdef{\cU}{\mathcal{U}}
\Mdef{\cV}{\mathcal{V}}
\Mdef{\cW}{\mathcal{W}}
\Mdef{\cX}{\mathcal{X}}
\Mdef{\cY}{\mathcal{Y}}
\Mdef{\cZ}{\mathcal{Z}}

\Mdef{\ca}{\mathcal{a}}
\Mdef{\ct}{\mathcal{t}}

\Mdef{\At}{\tilde{A}}
\Mdef{\Bt}{\tilde{B}}
\Mdef{\Ct}{\tilde{C}}
\Mdef{\Et}{\tilde{E}}
\Mdef{\Ht}{\tilde{H}}
\Mdef{\Kt}{\tilde{K}}
\Mdef{\Lt}{\tilde{L}}
\Mdef{\Mt}{\tilde{M}}
\Mdef{\Nt}{\tilde{N}}
\Mdef{\Pt}{\tilde{P}}

\Mdef{\tA}{\tilde{A}}
\Mdef{\tB}{\tilde{B}}
\Mdef{\tC}{\tilde{C}}
\Mdef{\tE}{\tilde{E}}
\Mdef{\tH}{\tilde{H}}
\Mdef{\tK}{\tilde{K}}
\Mdef{\tL}{\tilde{L}}
\Mdef{\tM}{\tilde{M}}
\Mdef{\tN}{\tilde{N}}
\Mdef{\tP}{\tilde{P}}

\Mdef{\ft}{\tilde{f}}
\Mdef{\xt}{\tilde{x}}
\Mdef{\yt}{\tilde{y}}

\Mdef{\Ab}{\overline{A}}
\Mdef{\Bb}{\overline{B}}
\Mdef{\Cb}{\overline{C}}
\Mdef{\Db}{\overline{D}}
\Mdef{\Eb}{\overline{E}}
\Mdef{\Fb}{\overline{F}}
\Mdef{\Gb}{\overline{G}}
\Mdef{\Hb}{\overline{H}}
\Mdef{\Ib}{\overline{I}}
\Mdef{\Jb}{\overline{J}}
\Mdef{\Kb}{\overline{K}}
\Mdef{\Lb}{\overline{L}}
\Mdef{\Mb}{\overline{M}}
\Mdef{\Nb}{\overline{N}}
\Mdef{\Ob}{\overline{O}}
\Mdef{\Pb}{\overline{P}}
\Mdef{\Qb}{\overline{Q}}
\Mdef{\Rb}{\overline{R}}
\Mdef{\Sb}{\overline{S}}
\Mdef{\Tb}{\overline{T}}
\Mdef{\Ub}{\overline{U}}
\Mdef{\Vb}{\overline{V}}
\Mdef{\Wb}{\overline{W}}
\Mdef{\Xb}{\overline{X}}
\Mdef{\Yb}{\overline{Y}}
\Mdef{\Zb}{\overline{Z}}

\Mdef{\db}{\overline{d}}
\Mdef{\hb}{\overline{h}}
\Mdef{\qb}{\overline{q}}
\Mdef{\rb}{\overline{r}}
\Mdef{\tb}{\overline{t}}
\Mdef{\ub}{\overline{u}}
\Mdef{\vb}{\overline{v}}

\Mdef{\hc}{\hat{c}}
\Mdef{\he}{\hat{e}}
\Mdef{\hf}{\hat{f}}
\Mdef{\hA}{\hat{A}}
\Mdef{\hH}{\hat{H}}
\Mdef{\hJ}{\hat{J}}
\Mdef{\hM}{\hat{M}}
\Mdef{\hP}{\hat{P}}
\Mdef{\hQ}{\hat{Q}}

\Mdef{\thetab}{\overline{\theta}}
\Mdef{\phib}{\overline{\phi}}

\Mdef{\uA}{\underline{A}}
\Mdef{\uB}{\underline{B}}
\Mdef{\uC}{\underline{C}}
\Mdef{\uD}{\underline{D}}

\Mdef{\bolda}{\mathbf{a}}
\Mdef{\boldb}{\mathbf{b}}
\Mdef{\bfD}{\mathbf{D}}

\Mdef{\fm}{\frak{m}}
\Mdef{\fp}{\frak{p}}

\Mdef{\eps}{\epsilon}

\newcommand{\cEi}{\cE^{-1}}
\newcommand{\cOcF}{\cO_{\cF}}
\newcommand{\cOcFH}{\cO_{\cF/H}}
\newcommand{\cOcFK}{\cO_{\cF/K}}
\newcommand{\cOcFL}{\cO_{\cF/L}}
\newcommand{\cOcFM}{\cO_{\cF/M}}

\newcommand{\cAt}{\cA_t}
\newcommand{\cAsft}{\cA^{sf}_t}
\newcommand{\All}{\mathrm{All}}
\newcommand{\siftyV}[1]{S^{\infty V(#1)}}
\newcommand{\efp}{E\cF_+}
\newcommand{\epp}{E\cP_+}
\newcommand{\piAts}{\pi^{\cA_t}_*}

\newcommand{\ExtAsft}{\Ext_{\cAsft(G)}}
\newcommand{\HomAt}{\Hom_{\cAt(G)}}
\newcommand{\ExtAt}{\Ext_{\cAt(G)}}
\renewcommand{\ft}{f}

\newcommand{\at}{a}
\newcommand{\att}{\tilde{a}}
\newcommand{\Vt}{V}
\newcommand{\htt}{\tilde{h}}

\newcommand{\piGs}{\pi^G_*}
\newcommand{\bbI}{\mathbb{I}}

\newcommand{\connsub}{\mathrm{ConnSub}}
\newcommand{\sub}{\mathrm{Sub}}
\newcommand{\cat}{\mathrm{Cat}}
\renewcommand{\mod}{\mbox{-mod}}
\newcommand{\Ehat}{\hat{E}}
\newcommand{\vark}{k_K}
\newcommand{\id}{\mathrm{id}}
\newcommand{\cAhatt}{\hat{\cA}_t}

\input{xypic}

\begin{document}
\title{Rational torus-equivariant stable homotopy V: the torsion Adams spectral sequence}
\author{J.P.C.Greenlees}
\address{Mathematics Institute, Zeeman Building, Coventry CV4, 7AL, UK}
\email{john.greenlees@warwick.ac.uk}
\date{}

\begin{abstract}
We provide a calculational method for rational stable
equivariant homotopy theory for a torus $G$ based on the homology of the Borel
construction on fixed points. More precisely we define an abelian torsion
model, $\cAt(G)$ of finite injective dimension, a homology theory
$\piAts(\cdot)$ taking values in $\cAt (G)$ based on the homology of
the Borel construction, and a finite Adams spectral sequence
$$\Ext_{\cAt(G)}^{*,*}(\piAts(X), \piAts(Y))\Rightarrow [X,Y]^G_*$$
for rational $G$-spectra $X$ and $Y$.
\vspace{2ex}
\begin{center}

\end{center}
\end{abstract}

\thanks{I am grateful to S.Balchin, M.Barucco, L.Pol and J. Williamson
for discussions on related projects and to EPSRC for support under EP/P031080/1.}
\maketitle

\tableofcontents

\section{Introduction}
\subsection{Context}
A wide range of methods are available for constructing models of
rational $G$-spectra and calculating with them. In various ways they
are based on assembling information from the geometric $H$-fixed
points as $H$ varies through the closed subgroups of $G$. From the
point of view of transformation groups, the most natural thing to use
as input for the contribution at $H$ 
is the homology of the Borel construction on the geometric $H$-fixed
points. We construct an abelian category $\cAt(G)$ from this data, call it the
{\em torsion model}, and prove that it gives an effective means of calculation.

 In the case of the circle group $G$ two methods are discussed in
 \cite{s1q}. The first is based on the abelian standard  model $\cA
 (G)$ and the second on the abelian torsion  model $\cAt (G)$, which is the
 rank 1 case of the model constructed here. Even in the rank 1 case, one can see the technical advantages of the standard model.
The standard model is of injective dimension
1 and monoidal, and the torsion model is of injective dimension 2 and
cannot be monoidal.   Because of this, the focus when considering other
groups has so far been on the standard model. Indeed,
the full form of the  torsion model has only been considered
previously in the case of the circle group.

Nonetheless, the torsion model also has some advantages. As described
above, 
the ingredients are very natural from the point of view of
transformation groups. In fact they are also rather natural from the
point of view of algebraic geometry, where they mirror the Cousin
complex. Partly for these reasons,  when identifying the algebraic model of a spectrum
in the standard model it is often useful to approach through the
torsion model. From the point of view of commutative algebra, 
 the standard model is based on complete and Noetherian objects and their
 localizations whilst the torsion model is based on torsion and
 Artinian objects. This often means that the vector spaces concerned
tend  to be smaller. For any of these reasons it is valuable to 
develop a torsion model. 

The purpose of the present paper is to document the abelian torsion
model $\cAt(G)$ and  its homological algebra when $G$ is a torus and to construct the Adams spectral 
sequence based on the torsion model. This is a preliminary step
towards showing that a model category closely related to differential graded objects
in $\cAt (G)$ is Quillen equivalent to the category of rational
$G$-spectra.  Ongoing work with Balchin, Pol and Williamson
considers torsion models of a similar type in the much more general context
of tensor triangulated categories. The paper \cite{Torsion1} is
concerned  with 1-dimensional Noetherian Balmer
spectrum, and a special case gives an actual torsion model in the rank
1 case. Work on higher dimensional Noetherian Balmer spectra is
underway, and it is hoped that it will provide a model which in the
particular case of rational $G$-spectra for a torus $G$ will be a model categorical
enhancement of the torsion abelian model considered here. 

\subsection{Main results}

We will define (Section \ref{sec:At}) an abelian category $\cAt(G)$ built from
torsion modules over the polynomial rings $H^*(BG/K)$ for all
subgroups $K$. The category  $\cAt(G)$ is rather easy to work with: it
has enough injectives (Section \ref{sec:algAt}), it is of  finite
injective dimension (Proposition \ref{prop:id}) and
it is straightforward to make calculations.  The category is 
 precisely designed to be the codomain of a
homology theory $\piAts : \Gspectra \lra \cAt(G)$ (Section \ref{sec:piAt}), and the main
theorem (Theorem \ref{thm:AtASS}) states that there is a finite Adams
spectral sequence 
$$\Ext_{\cAt (G)}^{*,*}(\piAts (X), \piAts (Y))\Rightarrow [X,Y]^G_*$$
strongly convergent for any rational $G$-spectra $X$ and $Y$.

\subsection{The series} This paper is Part V of a series providing 
algebraic methods for 
approaching  rational torus equivariant stable homotopy, but it does 
not depend mathematically or expositionally on Parts I-IV. 
Parts I and II \cite{tnq1,tnq2} set up and study an 
Adams Spectral Sequence based on the standard model. Part III 
\cite{tnq3} is a comparison of variants of the  models. Part IV \cite{tnq4} calculates 
the Balmer spectrum of finite spectra (it will not be 
published beyond the arXiv since results are subsumed in \cite{spcgq},
which covers all compact Lie groups).

\subsection{Notation}
The models assemble data from various subgroups, and it enormously
aids readability to have consistent and suggestive notation. The
ambient torus is $G$, and we generally let containment follow the
alphabet as in $G\supseteq H \supseteq K\supseteq L$. 
One of the features of rational $G$-spectra is that it is often convenient to
group together the data from all subgroups with the same identity
component. We often write
$\Kt$ for a subgroup with identity component $K$ and so forth. We also 
write $\cF$ for the family of finite subgroups of $G$ and $\cF/K$ for
the family of finite subgroups of $G/K$ (which is in bijection to the
set of subgroups $\Kt$ of $G$ with identity component $K$).

The piece of data corresponding to a subgroup $K$ is built from the fixed
point set on which $G/K$ acts, so we  write $V(G/K)$ and index on the
quotient group. We combine this with the above conventions, so that
$V(\cF/K)$ collects the data for $V(G/\Kt)$ for all subgroups $\Kt\in
\cF/K$.

As a general principle, abelian categories approximating rational $G$-spectra
are denoted $\cA (G)$, but with a subscript to indicate the type of
algebra to be used and a superscript to denote the geometric
isotropy. The absence of a  subscript indicates the standard model and
$\cAt(G)$ indicates the torsion model, with $\cAt^{\cK}(G)$ indicating
a torsion model for rational $G$-spectra with geometric isotropy in
$\cK$.

\subsection{Organization}
Section \ref{sec:semifreecircle} recapitulates the rank 1 semifree
case and should sensitize the reader to issues that will
arise. Section \ref{sec:At} defines the abelian torsion model. Section
\ref{sec:piAt} describes the homology functor on $G$-spectra taking
values in $\cAt (G)$. Section \ref{sec:algAt} begins the algebraic
study of $\cAt (G)$. Section \ref{sec:inj} identifies sufficiently
many injectives, giving a description involving local cohomology of
various localized polynomial rings and residue maps between them. 
In Section \ref{sec:id} we show that for a non-trivial torus,  the injective dimension of
the torsion model is $\leq 2r$.  Section \ref{sec:indcorep} explains how to pick
out the contribution from a specific subgroup by giving a 
corepresentation theorem. This then gives us all we need to construct
a finite and strongly convergent Adams spectral sequence based on the
homology of the Borel constructions of fixed points as stated: the
pieces are assembled in Section \ref{sec:AtASS}.

\section{The circle group}
\label{sec:semifreecircle}
The special case when $G$ is the circle group (i.e., the rank $r=1$)
was covered in \cite[Chapter 6]{s1q}. Nonetheless, it will be useful to run
through the arguments, partly because the treatment is a little
condensed in \cite{s1q}, and partly to motivate the general constructions we
will need later. 

To make the algebraic structures clearer, we work here with semifree
$G$-spectra (i.e., those with geometric isotropy in $\{1, G\}$) and
their model.  Thus only 
 $G$ and the trivial group $1$ will play a role and diagrams are much
 smaller. The case of general $G$-spectra  is covered along with other 
 ranks below. The basic algebraic ingredient is a graded commutative ring $\cOcF$
together with a multiplicative set $\cE_G$ playing the role of Euler
classes of representations of $G$.
This will be introduced in general later in a way that
will make the choice of notation clearer, but for semifree $G$-spectra
the ring is $k[c]$ for a field $k$ and an element 
$c$ of degree $-2$,  and the multiplicative set consists of the powers
of $c$. Accordingly, in this case $\cOcF \lra \cEi_G\cOcF$ is the map  $k[c]\lra
k[c,c^{-1}]$. For brevity we write $t=k[c,c^{-1}]$ (for Tate), and we will take 
$k=\Q$ for the topological applications.

\subsection{Homological algebra}
The objects of the semifree torsion abelian model $\cAsft (G)$ are
$X=(t\tensor V \stackrel{q}\lra T)$ where   $V$  is a $k$-module, $t=k[c,
c^{-1}]$, $T$ is a torsion $k[c]$-module and $q$ is a $k[c]$-map.
In view of the adjunction $\Hom_{k[c]}(t\tensor V,
  T)=\Hom_k(V, \Hom_{k[c]}(t,T))$, it  is sometimes more convenient to
consider the equivalent adjoint form  of the torsion
abelian category, with objects  $\tilde{X}=( V \lra T^t)$ where
$T^t=\Hom_{k[c]}(t, T)$. Note that maps on the torsion parts $A^t\lra
B^t$ are required to be of the form $\theta^t$ for a map $\theta :
A\lra B$, and cokernels are $\cok^t (A^t\lra B^t)=\cok(A\lra
B)^t$, which may not be the same as the ordinary cokernel of the module map
$A^t \lra B^t$.

The most naive way to construct objects of $\cAsft(G)$ from modules is
to form $\ft_G(V)=(t\tensor V \lra 0)$ and
$\ft_1(T)=(0\lra T)$. The former is
injective and for $T\neq 0$ the latter is not injective even if $T$ is an injective
$k[c]$-module.  (There is a similar but inequivalent construction with the same name in the
standard model, but there should be no confusion since we do not
consider the standard model in this paper). 

Alternatively, we may attempt to construct right adjoints $\at_H$ to evaluation
at $H$ for the relevant subgroups $H$. The evaluation of $X$ at $G$ is the $k$-vector space $V$ and
it  is easy to see that  $\ft_G$ is right adjoint to evaluation at
$G$, so $\at_G=\ft_G$. The evaluation of $X$ at the trivial subgroup
$1$ is the torsion
$k[c]$-module $T$, and we next describe its right adjoint $\at_1$. 
This is more obvious in the adjoint form where we have
$\att_1(T)=(id: T^t\lra T^t)$. Thus $\at_1(T)=(ev: t\tensor T^t \lra
T)$. It is immediate from the adjunction that $\at_1(I)$ is injective
if $I$ is an injective torsion $\cOcF$-module.

\begin{lemma}
The injective dimension of $\cAsft(G)$ is $\leq 2$. 
\end{lemma}

\begin{proof} Choose a resolution $0\lra T \lra I \lra
  J\lra 0$ of $T$ by torsion injective $k[c]$-modules. 

With $X=(t\tensor V\lra T)$ as before,  we take the maps (i)  $X \lra  \at_1(I)$, which is $T\lra I$ at $G/1$, and
(ii) $X \lra \at_G(V)$, which is the identity at $G/G$. This gives a
monomorphism
$X\lra \at_1(I)\oplus \at_G(V)$. The cokernel $C$ is $J$ at
$G/1$, and if we suppose it is $V'$ at $G/G$ we obtain a resolution 
$$0\lra X \lra \at_1(I)\oplus \at_G(V)\lra \at_1(J)\oplus \at_G(V')\lra
\at_G(V'')\lra 0. $$
\end{proof}

To see the injective dimension is exactly 2, we need to make a
calculation. 

\begin{lemma}
We have 
$$\ExtAt^s (\ft_G(k), \at_1(T))=
\dichotomy {\Ext_{k[c]}^{1}(t,T)& s=2}
{0 & \mbox{otherwise}}$$
\end{lemma}

\begin{proof}
If $T$ is a torsion $k[c]$-module with injective resolution 
$$0\lra T \lra I \lra J \lra 0$$ 
then Hom and Ext are given by the exact sequence 
$$0\lra T^t \lra I^t\lra J^t \lra \Ext_{k[c]}(t, T)\lra 0.$$
We may now write down a resolution of $\at_1(T)$:
$$0\lra \at_1(T) \lra \at_1(I)\lra \at_1(J)\lra 
\at_G(T^{\dagger t}) \lra 0$$
where $T^{\dagger t}=\Ext^1_{k[c]}(t,T)$. 
The answer is clear by applying $\HomAt(\ft_1(k), \cdot )$ since if 
$\tilde{Y}=(\tilde{q}: W\lra U^t)$ we have 
$$\HomAt(\ft_1(k), Y) =\ker (\tilde{q}: W\lra U^t) .$$
\end{proof}

One needs to think a little to identify a torsion module $T$ with 
$\Ext^1_{k[c]}(t, T)\neq 0$. However 
$$\Ext^1(t, T)= R\lim \left[ \cdots \stackrel{c} \lra 
\Sigma^{-4}T\stackrel{c}\lra \Sigma^{-2}T\stackrel{c}\lra T\right]$$
so we see $T=\bigoplus_{s\geq 1}\Sigma^{2s} k[c]/c^s$ will do. 

\begin{remark}
  \label{rem:suminjnotinj}
The sum of injectives need not be injective. Indeed, if we apply graded
vector space duality $(\cdot)^*$, 
we obtain
$$I=\left[ \bigoplus_s\Sigma^{2s}k[c]\right]^*=\prod_s\Sigma^{2s}k[c]^*,$$
we see that the map
$$\bigoplus_s \Hom (t, \Sigma^{2s}k[c]^*) \lra \Hom(t, \prod_s
\Sigma^{2s}k[c]^*)=\Hom (t, I)$$
is not an isomorphism. Choosing an element $\delta$ not in the image
(such as the map diagonal on nonzero entries in each degree) we see
that there is no solution to the problem
$$\xymatrix{0\rto &f_1(k[c]^*)\rto
  \dto^{\delta}&a_1(k[c]^*)\ar@{..>}[dl]\\
  &\bigoplus_s\Sigma^{2s}a_1(k[c]^*)&
  }$$
and hence $\bigoplus_s\Sigma^{2s}a_1(k[c]^*)$ is not injective. 
  \end{remark}




\subsection{The Adams spectral sequence}
Writing $\mbox{sf-$G$-spectra}$ for the homotopy category of semi-free
rational $G$-spectra, we begin by defining the homology theory 
$$\piAts: \mbox{sf-$G$-spectra} \lra \cAt^{sf} (G). $$
One may view this as a distillation of the power of the cohomology
of the Borel construction $H^*(EG\times_GX)$. From here on we will consistently use
the based version $H^*_G(X)=H^*(EG\times_G X, EG\times_G pt)$. We are
exploiting the fact that $H^*_G(S^0)=H^*(BG)=\Q[c]$ acts on
$H^*_G(X)$ and $H_*(EG\times_G X, EG\times_G pt)$; by degrees we
immediatly see that  the latter is always a torsion module.

To connect the topology and algebra, we consider the isotropy separation sequence
$$\efp \lra S^0 \lra \etf \lra \Sigma \efp$$
for the family $\cF$ of finite subgroups, and for a semifree $G$-spectrum $X$ we take
$$\piAt_*(X)=\pi^G_*(\etf \sm DEG_+ \sm X \lra \Sigma \efp  \sm DEG_+ \sm X). $$
Of course we are more precisely using 
$$\pi^G_*(\etf\sm DEG_+ \sm X)\cong t \tensor \pi_*(\Phi^GX)=:t\tensor V_X$$
and, since $X$ is semifree and $DEG_+  \sm EG_+ \simeq EG_+ $, 
$$\pi^G_*(\Sigma \efp \sm DEG_+ \sm X)\cong \pi^G_*(\Sigma EG_+\sm X)\cong 
H_*^{G}(\Sigma^2 X) =: T_X$$
(where $H^G_*$ denotes homology of the Borel construction) so that 
$$V_X=H_*(\Phi^GX)$$
and 
$$T_X=\Sigma^2 H_*^G(X). $$

The functor $\piAts$ provides the $d$-invariant
$$d=\piAts: [X,Y]^G_*\lra \Hom_{\cAt(G)}(\piAts (X), \piAts(Y)). $$
We will now proceed by the usual method towards an Adams spectral
sequence. 

The first question is whether we can realise enough
injectives. We have made a start since  $\at_G(k)=\ft_G(k)=\piAts(\etf)$. Next we 
would like to realize $\at_1(k[c]^*)$. We could use Brown representability as in 
the general case below (Lemma \ref{lem:rinj}), but by way of
variation, we give here a construction in terms of 
well known objects. 

It is standard  that $\pi^G_*(EG_+)=H_*^G(\Sigma EG_+)=\Sigma k[c]^*$
and therefore (care about
suspensions) $\piAts(EG_+)=\Sigma^2 \ft_1(k[c]^*)$. 
We break $\at_1(k[c]^*)$  down using the short exact sequence
$$0\lra \ft_1(k[c]^*)\lra \at_1(k[c]^*)\lra \ft_G((k[c]^*)^t)\lra
0. $$
It is reassuring to note that $(k[c]^*)^t\cong t$, but in fact it is
more helpful to retain the functional form $(k[c]^*)^t$.  In any case, we see that $\at_1(k[c]^*)$ is the
fibre of  a map  
$$\ft_G(t)\lra \Sigma \ft_1(k[c]^*), $$
so (with care again about suspensions) the realization should be the fibre of a map 
$$\etf [t]\lra \Sigma^{-1}EG_+,  $$
where $[t]$ indicates the use of a module of graded coefficients. 
We may calculate the maps $\etf\lra EG_+$ using the  exact
sequence
$$\cdots \lra [\Sigma EG_+, EG_+]^G\lra [\etf, EG_+]^G\lra [S^0,
EG_+]^G\lra \cdots$$
A priori it is long exact, but the outer two groups are in odd
degrees, so it is short exact. Since $S^{\infty z}$ is a smash factor of
$\etf$, multiplication by $c$ is an
isomorphism in the middle we see $[\etf, \Sigma^{-1}EG_+]^G=t$. 
Indeed, we may take the map
$$t\tensor (k[c]^*)^t \lra \Sigma^{-1}k[c]^*$$
to be evaluation. Following through the isomorphisms we see
$$a_1(k[c]^*)=\piAts( \fibre (ev: \etf[t]\lra \Sigma^{-1}EG_+)). $$

\begin{thm} \cite[Theorem 6.6.2]{s1q}
For rational semifree $G$-spectra $X$ and $Y$ there is an Adams spectral sequence
$$\ExtAsft^{*,*} (\piAts (X), \piAts (Y))\Rightarrow [X,Y]^G_*. $$
This is a strongly convergent spectral sequence which collapses at $E_3$.
\end{thm}

\begin{proof}
As usual we need only show that enough injectives are realizable and
that the $d$-invariant is an isomorphism when $Y$ is one of the
realizable injectives. Convergence is clear since $\piAts$ is detects
contractibility and commutes with telescopes. 

Now we need to show that 
$$[X, \at_1(k[c]^*)]^G_*\lra \HomAt (\piAts (X),  \at_1
(k[c]^*)) =\Hom_{k[c]}(T_, k[c]^*) =\Hom_{k}(T_X, k)$$
is an isomorphism,  where $X=(t\tensor V_X\lra T_X)$.
This is straightforward  since we can see immediately that $[f_G(k),
\at_1(k[c]^*)]^G_*=0$. That in turn means we may assume $X$ is free
and hence it suffices to take $X=G_+$. Now we only need to check that 
$$d: [G_+, \ft_1(k[c]^*)]^G_*=[G_+, EG_+]^G_*\lra
\Hom_{k[c]}(\pi^G_*(G_+), \pi^G_*(EG_+))$$
is an isomorphism.  The idea is to use the cofibre sequence $G_+\lra
S^0\lra S^z$ and connectivity: this is implemented in \cite[Lemma 6.3]{gfreeq}. 
\end{proof}

\section{The torsion model}
\label{sec:At}

We now begin work on the general case, so that $G$ is a torus of rank $r$.
We will write down the category $\cAt (G)$ directly, 
because this emphasizes the algebraic simplicity but many features 
will appear mysterious.  We will return to explain the form of the 
definition in Section \ref{sec:piAt}: the category is precisely 
designed as the appropriate receptacle for a torsion-based homology 
theory on rational $G$-spectra. 

\subsection{Inflation systems and Euler classes}
We begin with a diagram of rings and some localizations. For the
present we will restrict this to the context of our applications. 

Starting with our torus $G$, we consider the poset of {\em connected}
subgroups $H$ ordered by inclusion. 
We then have a diagram 
$$\cO_{\cF /}: \connsub(G)^{op}\lra \crings$$
to graded commutative rings. If $K \subseteq H$ the map $\cOcFH\lra
\cOcFK$ is called {\em inflation}, so we call $\cO_{\cF/}$ the {\em
  inflation diagram} of rings. 

For each $K\subseteq H$ we have a multiplicatively closed subset
$\cE_{H/K}\subseteq \cOcFK$, and these are compatible in the sense
that, for $H \supseteq K \supseteq L$, the inflation of $\cE_{H/K}\subseteq \cOcFK$ lies in
$\cE_{H/L}\subseteq \cOcFL$, and in fact $\cE_{H/L}$ is generated by $
\cE_{H/K}$ and $ \cE_{K/L} .$

An $\cOcFK$-module $M$ is {\em torsion} if $\cEi_{H/K}M=0$ for every
$H> K$.

There are two examples to bear in mind from equivariant homotopy
theory.  The second is the motivating example, and the default interpretation.
The notation we have used for an abstract inflation functor with Euler classes comes from
it. The occurrence of many subgroups is sometimes found to be a
distraction, and the author finds that the first example is often helpful as a 
warm-up.
 
\begin{example}
  \label{eg:simple}
The connected group inflation functor is the diagram of rings with  
$\cOcFK=H^*(BG/K)$ and $\cE_H=\{ e_1(V)\st V^H=0\}$. Here $e_1(V)\in
H^{|V|}(BG)$ is the classical Euler class. 
\end{example}

\begin{example}
\label{eg:fullisotropy}
The full isotropy topological example has 
$\cOcFK=\prod_{\Kt} H^*(BG/\Kt)$ (with the product over all subgroups
$\Kt$ with identity component $K$), and 
$\cE_H=\{ e(V)\st V^H=0\}$. Here $e(V)\in \cOcF$ has $F$-component
the Euler class $e_F(V)=e_1(V^F)\in H^{|V^F|}(BG/F)$. 

The ring $\cOcFK$ has one idempotent $e_{\Kt}$ for each subgroup ${\Kt}$, and
hence any $\cOcFK$-module $M$ has summands $e_{\Kt}M$. The nature of
the Euler classes means that any torsion module  is a sum of these
pieces:  $M=\bigoplus_{\Kt} e_{\Kt}M$.

If $K$ is a subgroup of $H$ the inflation map $\cOcFH\lra \cOcFK$
requires the observation that for each subgroup $\Kt$ with identity
component $K$ there is a unique subgroup $\Ht$ (namely $\Ht=\Kt\cdot
H$) with with identity component $H$ with $\Kt $ cotoral in $\Ht$.  
\end{example}

As indicated above,  the reader should always assume we are dealing with full
isotropy (i.e., Example \ref{eg:fullisotropy}), but may find it useful
to think about Example \ref{eg:simple} for a less cluttered introduction.

\subsection{The definition}
In the presence of an inflation diagram of rings, and an Euler system
of multiplicatively closed subsets we may define a torsion category. 

\begin{defn}
Objects of the abelian torsion model $\cAt(G)$ are cochain complexes
\begin{multline*}
\cEi_G\cOcF \tensor V(\cF/G)\stackrel{h^0}\lra 
\bigoplus_{\codim(H)=1}\cEi_H\cOcF \tensor_{\cOcFH} V(\cF/H)\stackrel{h^1}\lra \\
\bigoplus_{\codim(K)=2}\cEi_K\cOcF \tensor_{\cOcFK} V(\cF/K)\stackrel{h^2}\lra \cdots 
\stackrel{h^{r-2}}\lra 
\bigoplus_{\codim(L)=r-1}\cEi_L\cOcF \tensor_{\cOcFL} V(\cF/L)\stackrel{h^{r-1}}\lra 
V(\cF/1) 
\end{multline*}
where the sums are over connected groups of the stated sort and $V(\cF/K)$ is a torsion $\cOcFK$-module. 
The decomposition into direct sums and tensor products 
is a given part of the structure so the morphisms of $\cAt(G)$ are
given by $\cOcFK$-maps $V(\cF/K)\lra V'(\cF/K)$ that give a map of cochain
complexes. 

It is convenient to formalize this a little further. 
\begin{itemize}
\item For each connected subgroup $K$ we have a torsion $\cOcFK$-module
  $V(\cF/K)$, which we call the {\em $\cF/K$ torsion component} of the object. 
\item If $K\supseteq L$ there is an inflation map $\cOcFK \lra \cOcFL$
  and an upward {\em vertical} map
$$v_{\cF/K}^{\cF/L}: V(\cF/K)\lra \cEi_{K/L} \cOcFL\tensor_{\cOcFK}V(\cF/K)$$
of $\cOcFK$-modules.
\item If $K\supseteq L$ there is an inflation map $\cOcFK \lra \cOcFL$
  and rightward  {\em horizontal} structure maps
$$h_L^K: \cEi_K \cOcFL \tensor_{\cOcFK} V(\cF/K) \lra V(\cF/L)$$
of $\cOcFL$-modules
\item For each connected subgroup $M$ the horizontal maps for
  subgroups containing $M$ form  a cochain complex $C_{G/M}(\cF/M)$
\begin{multline*}
\cEi_{G/M}\cOcFM \tensor V(\cF/G)\stackrel{h^0_{G/M}}\lra 
\bigoplus_{\codim(H/M)=1}\cEi_{H/M}\cOcFM \tensor_{\cOcFH} V(\cF/H)\stackrel{h^1_{G/M}}\lra \\
 \cdots 
\lra \bigoplus_{\codim(L/M)=\dim(G/M)-1}\cEi_{L/M}\cOcFM
\tensor_{\cOcFL} V(\cF /L)
\stackrel{h^{\dim(G/M)-1}_{G/M}}\lra V(\cF/M) 
\end{multline*}
of $\cOcFM$-modules.
\end{itemize}
\end{defn}

\begin{remark}
\newcommand{\LI}{\mathcal{LI}}
(i) We note that there are no suspensions on the objects $V(\cF /H)$. At
present it seems odd to even comment on this, but it will be the basis
of discussion when we return to consider $G$-spectra. 

(ii) There is no direct relationship between the different modules
$V(\cF/K)$. We can define a localized inflation diagram $\LI$ of module categories 
$$\LI_G: \connsub (G)\lra \cat$$
where $\LI_G(G/H)=\cOcFH\mod$ and  if $L\subseteq K$ then 
$$\LI_G(\pi_{G/K}^{G/L}): \LI_G(G/K)=\cOcFK\mod \lra \cOcFL\mod =\LI_G(G/L)$$
is defined by 
$$\LI_G(\pi_{G/K}^{G/L}) (M) =\cEi_{K/L}\cOcFL\tensor_{\cOcFK}M. $$
In this context $V$ is a section of the diagram  $\LI_G$. 
\end{remark}

\begin{example}
When $G$ is a circle,  the diagram is very simple:  an object is given
by a diagram
$$\xymatrix{
\cEi_G \cOcF \tensor V(\cF/G)\rto & V(\cF/1)\\
V(\cF/G)\uto &
}$$
where $V(\cF/G)$ is a graded $\Q$-vector space and $V(\cF/1)$ is a torsion $\cOcF$-module.
\end{example}

\begin{example}
In rank 2 we may still  display the diagram.

In short form, $\cAt (G)$ is the category of cochain complexes of $\cOcF$-modules 
$$\cEi_G\cOcF \tensor V(\cF/G) \stackrel{h^0}\lra \bigoplus_H\cEi_H \cOcF 
\tensor_{\cOcFH}V(\cF/H) \stackrel{h^1}\lra  V(\cF/1), $$
where $V(\cF/G)$ is a $\Q$-module, $V(\cF/H)$ is a torsion $\cOcFH$-module
and $V(\cF/1)$ is a torsion $\cOcF$-module. 
 More explicitly objects are actually part of a larger  diagram with this as the top horizontal: 
$$
\xymatrix{
\cEi_G\cOcF \tensor V(\cF/G)\rto^(.4){h^0} & \bigoplus_H\cEi_H \cOcF 
\tensor_{\cOcFH} V (\cF/H)\rto^(.7){h^1} &  V(\cF/1)\\
\cEi_{G/H}\cOcFH \tensor V(\cF/G) \uto^j \rto^(.6){h^G_H} & V(\cF/H) \uto^j & \\
V(\cF/G) \uto^i &&
}$$
The top row is the cochain complex $C_G(\cF/1)$ of $\cOcF$-modules, the
second row (which is really just one of the countably many second rows
corresponding to codimension 1 connected subgroups $H$) is the cochain complex $C_{G/H}(\cF/H)$
of $\cOcFH$-modules and the third row is the cochain complex $C_{G/G}(\cF/G)$ of 
$\Q$-modules, since $\cO_{\cF/G}=\Q$.  We require that $V(\cF/G)$ is a
$\Q$-module,  
 $V(\cF/H)$ is a torsion  $\cOcFH$-module in the sense that
 $\cEi_{G/H}V(\cF/H)=0$ and $V(\cF/1)$ is a torsion  $\cOcF$-module in the
 sense that $\cEi_H V(\cF/1)=0$ for all circle subgroups $H$. 
\end{example}

\subsection{Maps into sums}

 Since the map from a sum to a product is a monomorphism, a map 
$\theta: A\lra \bigoplus_i B$ is determined by the components 
$\theta_i: A\lra B_i$. If $A$ is finitely generated, only finitely 
many of these are non-zero, but in general 
$$\Hom (A, \bigoplus_i B_i) =\ilim_{\alpha} \Hom (A_\alpha, 
\bigoplus_i B_i)
=\ilim_{\alpha} \bigoplus_i \Hom (A_\alpha,  B_i), $$
where $A_\alpha$ runs through finitely generated submodules of $A$. 
We say that $\{ \theta_i\}_i$ is {\em locally finite} if it lies in
this subgroup. 

We may consider what this means for the horizontal maps
$$h: \bigoplus_K \cEi_K \cOcFL \tensor_{\cOcFK} V(\cF/K) \lra \bigoplus_L
V(\cF/L). $$
Such a map $h$ is freely and uniquely determined by the components
$$h(K) : \cEi_K \cOcFL \tensor_{\cOcFK} V(\cF/K) \lra \bigoplus_L
V(\cF/L)$$
by the universal property of the first sum. This map $h(K)$ in turn is
determined by the maps
$$h_L^K: \cEi_K \cOcFL \tensor_{\cOcFK} V(\cF/K) \lra 
V(\cF/L)$$
but for each fixed $K$, the collection $h^K_*$ (where * runs through
the subgroups $L\subseteq K$) is  subject to the
condition of being locally finite. 

In general it is probably easier to consider the map $h$ as a whole
rather than decomposing it into factors. 

\subsection{Extra idempotents}
The full-isotropy topological example of Example
\ref{eg:fullisotropy} has the feature that the ring
$\cOcFK$ contains   idempotents for each subgroup $\Kt$ with identity
component $K$. It is a consequence of the torsion condition that the
natural map from the sum of idempotent pieces gives a direct sum decomposition
$$V(\cF/K)=\bigoplus_{\Kt}\Vt (G/\Kt),$$
with $\Vt (\Kt)$ a torsion $H^*(BG/\Kt)$-module. This is the reason
for our notation, since of course we usually find $V (\cF/ K)\neq
V(G/K)$. 

This also allows us to explain the quotient notation: for comparison
with subgroups we refer to quotients call it $V(G/H)$ after the
ambient group. The point is that  if $G\supseteq H \supseteq M$ we have a 
canonical isomorphism $(G/M)/(H/M)=G/H$.

This means that there is a second layer of sums available for
decomposition, and we may consider the idempotent pieces
$\htt_{\Lt}^{\Kt}$.  These  also determine the map $h$, but these  are
now subject to two separate sets of local finiteness conditions. 

In fact the first condition is that $\htt_{\Lt}^{\Kt}$ is only nonzero
when $\Lt$ is cotoral in $\Kt$. Then for a fixed subgroup $\Kt$, the
collection $\htt^{\Kt}_*$ (where $*$ runs through subgroups $\Lt$
cotoral in $\Kt$) is locally finite. Furthermore, for each fixed $K$
we may write $h^K_{\Lt}=\bigoplus_{\Kt}\htt^{\Kt}_{\Lt}$ and then the
collection $h^K_*$ is locally finite.

\subsection{Support and geometric fixed points}

The most visible feature of an object of $\cAt(G)$ is where it is
non-zero. 

\begin{defn}
For an object $X$ of $\cAt(G)$ the {\em connected support} is defined by 
$$\supp_c (X)=\{ H \in \connsub(G) \st V(\cF/H)\neq 0\}.$$

In the full isotropy example we may also define the {\em support}
$$\supp (X)=\{ \Ht \in \sub(G) \st \Vt (G/\Ht)\neq 0\}.$$
\end{defn}

\begin{remark} It is clear that in the full isotropy example the
  connected support can be recovered from the support 
$$\supp_c(X)=\{ H \st \mbox{ there is a subgroup } \Ht \in \supp
(X) \mbox{ with identity component } H \}.$$
\end{remark}

For a connected subgroup of $G$ there is a functor 
$$\Phi^K: \cAt(G)\lra \cAt (G/K). $$
obtained by picking out the part of the diagram below $K$ in the sense
that if $H\supseteq K$ 
$$(\Phi^K X) ((G/K)/(H/K))=X(G/H). $$

\begin{example}
For example if
$$X=[\cEi_G\cOcF \tensor V(\cF/G) \stackrel{h^0}\lra \bigoplus_H\cEi_H \cOcF 
\tensor_{\cOcFH}V(\cF/H) \stackrel{h^1}\lra  V(\cF/1)]$$
we have
$$\Phi^KX=\left[ \cEi_{G/K}\cOcFK\tensor V(\cF/G)\lra V(\cF/K) \right]. $$
\end{example}

It is evident that 
$$\supp (\Phi^KX)=\supp (X)\cap \{ H \st H\supseteq K\}, $$
where we identify the subgroups of $G/K$ as subgroups of $G$ containing $K$.

\section{The torsion model and the torsion homology functor}
\label{sec:piAt}
 
We return to the topology which motivated the definition of
$\cAt(G)$. Thus $G$ is a torus of rank $r$ and we consider $G$-spectra
 with arbitrary geometric isotropy. In this section we will define 
the homology functor $\piAts: \Gspectra \lra \cAt (G)$. 

\subsection{Isotropy separation}
First we recall the filtration 
$$\emptyset \subset \cF_{\leq 0}\subset \cF_{\leq 1}\subset
\cF_{\leq 2} \subset \cdots  \subset \cF_{\leq r}=\All$$
of the set of closed subgroups, where $\cF_{\leq s}=\{ K \st \dim (K)\leq s\}$. Taking universal
spaces we have the diagram 
$$\xymatrix{
\ast =E(\emptyset)_+ \rto & E(\cF_{\leq 0})_+\rto \dto  &E(\cF_{\leq 
  1})_+\rto \dto &E(\cF_{\leq 
  2})_+\rto \dto 
&\cdots \rto&E(\cF_{\leq r})_+=E\All_+=S^0 \dto \\
&E\lr{0}&E\lr{1}&E\lr{2}&&E\lr{r}
}$$
where $E\lr{s}=\cofibre(E(\cF_{\leq r-1})_+\lra E(\cF_{\leq r})_+)$.

Composing the vertical maps with the connecting maps, we obtain the
sequence of maps
$$\xymatrix{ 
E\lr{r}\rto&\Sigma E\lr{r-1}\rto &
\cdots \rto&\Sigma^{r-1} E\lr{1}\rto &\Sigma^r E\lr{0}. 
}$$ 
This is a cochain complex in the sense that the composite of two
adjacent maps is nullhomotopic.

Finally, by use of idempotents in Burnside rings, we have a rational splitting
$$E\lr{s}\simeq \bigvee_{\dim (\Kt)=s}E\lr{\Kt} \simeq \bigvee_{\dim
  (K)=s} \siftyV{K}\sm E\cF/K_+ , $$
where $\Kt$ runs through all subgroups of dimension $s$ and $K$ runs
through connected subgroups of dimension $s$. As usual, 
$E\lr{\Kt}=\cofibre(E[\subset \Kt]_+\lra E[\subseteq \Kt]_+)$ (with
geometric isotropy the singleton $\{ \Kt\}$), and
$\siftyV{K}=\bigcup_{V^K=0}S^V$ (with geometric isotropy consisting of
exactly those 
subgroups containing $K$). 

\begin{example}
If $r=2$ we may write this in more familiar terms 
$$\xymatrix{
\ast \rto & \efp \rto \dto  &\epp \rto \dto &S^0\dto \\
&\efp &\bigvee_H \siftyV{H}\sm E\cF/H_+&\siftyV{G}
}$$
where $\cF$ is the family of finite subgroups and $\cP$ is the family
of proper subgroups and $H$ runs through circle subgroups.  This in turn gives a sequence 
$$\siftyV{G}\lra \bigvee_H \Sigma \siftyV{H}\sm E\cF/H_+ \lra \Sigma^2 \efp$$
in which the composite is null. 
\end{example}

\subsection{Homotopy of pure strata}
The first approximation to $\cAt (G)$ is obtained by taking the 
homotopy of the isotropy separation filtration. One can give a formula for the homotopy of
the subquotients in terms of the homology of the Borel construction. 

\begin{lemma}
$$\piGs(E\lr{s}\sm
X)=\Sigma^{r-s}\bigoplus_{\dim(\Kt)=s}H^{G/\Kt}_*(\Phi^{\Kt}X)$$
where the sum is over all subgroups $\Kt$ of dimension $s$. The term 
$H^{G/\Kt}_*(\Phi^{\Kt}X)$ is a torsion $H^*(BG/\Kt)$-module.
\end{lemma}

The $E_1$-term of the spectral sequence of the filtration is the
homotopy of the sequence
$$\xymatrix{ 
E\lr{r}\sm X\rto&\Sigma E\lr{r-1}\sm X\rto &
\cdots \rto&\Sigma^{r-1} E\lr{1}\sm X\rto &\Sigma^r E\lr{0}\sm X  
}$$ 
namely 
$$\xymatrix{
\piGs(E\lr{r}\sm X)\rto \ar@{=}[d]
&\piGs(\Sigma E\lr{r-1}\sm X)\rto \ar@{=}[d]&
\cdots \rto&\piGs(\Sigma^r E\lr{0}\sm X) \ar@{=}[d]\\
H_*(\Phi^G X)\rto &
\Sigma^2  \bigoplus_{\dim(\Ht)=r-1}H^{G/\Ht}_*(\Phi^{\Ht} X)\rto &
\cdots \rto&
\Sigma^{2r}  \bigoplus_{\dim(F)=0}H^{G/F}_*(\Phi^{F} X)
}$$
In geometric terms this is the direct analogue of the Cousin complex.  In order to get
something more algebraic we need to smash the whole thing with $D\efp$
first. 

\begin{lemma}
\label{lem:htpypureDefp}
$$\piGs(E\lr{s}\sm D\efp \sm X)=\Sigma^{r-s}\bigoplus_{\dim(\Kt)=s}
\cEi_K\cOcF\tensor_{H^*(BG/\Kt)} H^{G/\Kt}_*(\Phi^{\Kt}X), $$
where the sum is over all subgroups $\Kt$ of dimension $s$. The term 
$H^{G/\Kt}_*(\Phi^{\Kt}X)$ is a torsion $H^*(BG/\Kt)$-module.
\end{lemma}

\begin{remark}
If $T$ is a module over $H^*(BG/\Kt)$ then 
$$\cOcF \tensor_{H^*(BG/\Kt)}T=[e_{\Kt}\cOcF ]\tensor_{H^*(BG/\Kt)}T$$
where $e_{\Kt}$ is the idempotent supported on finite subgroups
cotoral in $\Kt$. 
\end{remark}

\subsection{Collecting subgroups by identity component}
In this section so far, we have treated all subgroups of the same dimension 
equally. However the behaviour of Euler classes (specifically the fact
that $e(\alpha^n)=ne(\alpha)$, and we are working rationally) means that it becomes 
important to collect together the subgroups according to their 
identity component.  The convention is that letters $G, H, K, ...$
will be connected subgroups, and $\Ht$ is a subgroup with identity 
component $H$, $\Kt$ is a subgroup with identity component $K$, and so
forth. 

\subsection{The homology functor}
It should be apparent that the structures in $\cAt(G)$ mirror those in
topology. The following definition should therefore not be a
surprise. 

\begin{defn}
$$\piAts: \Gspectra \lra \cAt (G)$$
is defined by taking  $C_G(\cF)$ to be $\piGs (D\efp \sm \cdot)$
applied to the sequence 
$$\xymatrix{
E\lr{r}\sm X\rto&\Sigma E\lr{r-1}\sm X\rto &
\cdots \rto&\Sigma^{r-1} E\lr{1}\sm X\rto &\Sigma^r E\lr{0}\sm X 
}$$
Specifically
$$V_X(\cF/H)=\bigoplus_{\Ht}V(G/\Ht)$$ and
$$V_X(G/\Ht)=\pi^{G}_* (\Sigma^{\dim (G/H)} E\lr{\Ht} \sm  X) 
=\Sigma^{2\dim(G/H)}H^{G/\Ht}_*(\Phi^{\Ht}X)$$ 
\end{defn}

\begin{lemma}
This definition does give a functor to $\cAt(G)$.
\end{lemma}

\begin{proof}
 The fact that the modules are torsion and the maps are of the correct
 form is the content of Lemma \ref{lem:htpypureDefp}. 
It is clear that $C_G(\cF)$ is a cochain complex since the composite of two
morphisms is null-homotopic.
Finally, we need to observe that we have the appropriate comparison
map $C_{G/K}(\cF/K)\lra C_{G/L}(\cF/L)$ when $L\subseteq K$. It evidently suffices to
treat the case $L=1$, and at the point $G/K$  we need the vertical map 
$$\xymatrix{
\cEi_K \cOcF\tensor_{\cOcFK}V(\cF/K) \rrto^{\cong}&&\piGs 
(\siftyV{K}\sm D\efp \sm E\cF/K_+\sm  X)  \\
V(\cF/K) \rto^{\cong}\uto & \pi^{G/K}_*(E\cF/K_+\sm \Phi^K X) \rto^{\cong}&
\piGs 
(\siftyV{K}\sm E\cF/K_+\sm  X) \uto
}$$
It is apparent this is induced by the map $S^0\lra D\efp$.

All of the justifications are natural in $X$, so this does give a
functor. 
\end{proof}

\begin{remark}
Note the conventions on suspensions. In transformation groups,  the modules $H_*^{G/K}(\Phi^KX)$
occur naturally. We could have arranged that they occur in the model 
without suspensions, but for general groups this would lead to 
confusion.  Instead, with our conventions, we find that in $\piAts(X)$
we have 
$$\Vt_X (G/\Kt)=\Sigma^{2\dim (G/\Kt)} H^{G/\Kt}_*(\Phi^{\Kt}X)), $$
so that the contribution from a subgroup occurs suspended by twice its codimension. 

For example in rank 1, we take the homotopy of 
$$D\efp \sm \etf \sm X \lra \Sigma \efp \sm X$$
giving the object 
$$\cEi_G\cOcF \tensor H_*(\Phi^GX)\lra 
\bigoplus_F\Sigma^2H_*^{G/F}(\Phi^FX)$$
Thus we have 
$$V_X(G/G)=H_*(\Phi^GX), 
V_X(G/1)=\Sigma^2\bigoplus_FH^{G/F}_*(\Phi^FX). $$
\end{remark}

\begin{remark}
One of the attractive features of the torsion model is that it
directly reflects the geometric isotropy. It is immediate that the
geometric isotropy of $X$ is given by 
$$\cI_G(X)=\supp (\piAts (X)). $$
\end{remark}

\section{Algebra of the  torsion model}
\label{sec:algAt}
We now turn to an algebraic study of the torsion model. Since our
purpose is to define an Adams spectral sequence, it is not surprising
that this is mostly about homological algebra.

\subsection{Skyscraper objects}
We begin by giving a name to the most obvious and elementary
construction (As in Section \ref{sec:semifreecircle}, this is inequivalent to the construction of
the same name in the standard model; there is little danger of
confusion since the standard model does not feature in this paper). 

\begin{defn}
If $T$ is a torsion $\cOcFH$-module, we wrtite $\ft_H(T)$
for the object
$$\ft_H(T)(\cF/K)=
\dichotomy{\cEi_H\cOcF \tensor_{\cOcFH} T& \mbox{ if } K=H}
{0&  \mbox{ if } K\neq H}
$$
\end{defn}

Evidently if $T\neq 0$,  the object $\ft_H(T)$ corresponds to a $G$-spectrum with
connected geometric isotropy $\{H\}$. In particular
$$\piAts (E\langle \Ht\rangle)=\ft_H(\Sigma^{\dim (G/\Ht)}H_*(BG/\Ht)), $$
so their importance is clear. However, even though $H_*(BG/\Ht)$ is an
injective $\cOcFH$-module,  from an algebraic point of view this
object is not particularly simple. 

\begin{lemma}
For a torsion $\cOcFK$-module $T$, 
$$\HomAt ( X, \ft_K(T))=\Hom_{\cOcFK }(C_KX , T)$$
where 
$$C_KX=\cok (\bigoplus_{H>K} \cEi_{H/K}\cOcFK\tensor_{\cOcFH}V_X(\cF/H) \lra V_X(\cF/K))$$
\end{lemma}

We see that $C_GX=V_X(\cF/G)$, so that $\ft_G(W(\cF/G))$ is always
injective, but for $H\neq G$ the object
$\ft_H(T)$ is  never injective unless it is zero.  

\subsection{Adjoint form} 
It is invaluable to have a right adjoint to evaluation at a particular
subgroup $H$, and in writing these down we would like to work with
adjoint form of torsion diagrams. The basic idea is as in the rank 1
case, but this is complicated by the fact that there are
infinitely many connected subgroups in most dimensions.

\begin{example}
We make this explicit in the case $G$ has rank 2. Suppose then that 
$$X=\left[ \cEi_G \cOcF\tensor V\stackrel{h^0}\lra\bigoplus_H \cEi_H\cOcF 
\tensor_{\cOcFH} T(H)\stackrel{h^1}\lra S\right]$$
where $V$ is a $\Q$-vector space $T(H)$ is a torsion $\cOcFH$-module 
and $S$ is a torsion $\cOcF$-module (against our general principle, 
we are abbreviating $T(\cF/H)=T(H))$. We consider first the case when
only finitely many of the $T(H)$ are non-zero.  The adjoint form is 
then  
$$\xymatrix{
V\rto &\bigoplus_H \Hom_{\cOcFH}(\cEi_{G/H}\cOcFH, T(H)) \rto 
&\Hom_{\cOcF} (\cEi_G\cOcF, S) 
}$$

In long form, $X$ is given by a diagram 
$$
\xymatrix{
\cEi_G\cOcF \tensor V\rto^(.37){h^0} & \bigoplus_H\cEi_H \cOcF 
\tensor_{\cOcFH} T(H) \rto^(.78){h^1} &  S\\
\cEi_{G/H}\cOcFH \tensor V \uto^j \rto^(.56){h^G_H} & T(H) \uto^j & \\
V\uto^i &&
}$$

We convert this to adjoint form 
$$\xymatrix{
&&S\dto \\
&\bigoplus_H T(H)\rto \dto &  \Hom_{\cOcF} (\cEi_H\cOcF, S)\dto \\
V\rto &\bigoplus_H \Hom_{\cOcFH}(\cEi_{G/H}\cOcFH, T(H)) \rto &\Hom_{\cOcF} (\cEi_G\cOcF, S) 
}$$
where, for the purpose of understanding the bottom right entry, it is worth noting 
\begin{multline*}
\Hom_{\cOcF}(\cEi_G\cOcF, S) 
=\Hom_{\cOcF}(\cEi_H\cOcF\tensor_{\cOcFH}\cEi_{G/H}\cOcFH, S) =\\
\Hom_{\cOcFH}(\cEi_{G/H}\cOcFH, \Hom_{\cOcF}(\cEi_H\cOcF, S)) 
\end{multline*} 

When infinitely many of the $T(H)$ are non-zero, the map from $V$ maps
into the product, but it is a locally finite map. Accordingly, it is
still sufficient to describe the second map as coming from the sum. 
\end{example}

It is possible in principle to describe the adjoint form in
full detail. However our narrow purpose is to define a right adjoint
to evaluation, so  we will just
record what the adjoint form does on components. 

\begin{defn}
Given a flag
$$G\supset H_{r-1} \supset H_{r-2} \supset \cdots \supset
H_1 \supset H_0=1$$
of connected subgroups, the corresponding component of the torsion
model is the cochain complex 
\begin{multline*}
\cEi_G \cOcF \tensor V(\cF/G)\lra
\cEi_{H_{r-1}} \cOcF \tensor_{\cO_{\cF/H_{r-1}}}  V(\cF/H_{r-1})\\
\lra
\cEi_{H_{r-2}} \cOcF \tensor_{\cO_{\cF/H_{r-2}}}  V(\cF/H_{r-2})\lra
\cdots \lra 
\cEi_{H_{1}} \cOcF \tensor_{\cO_{\cF/H_{1}}}  V(\cF/H_{1})\lra
  V(\cF/1)
\end{multline*}
Its adjoint form is 
\begin{multline*}
V(\cF/G)\lra
\Hom_{\cO_{\cF/H_{r-1}}}(\cEi_{G/H_{r-1}}\cO_{\cF_{H_{r-1}}}, 
V(\cF/H_{r-1})) \\
\lra \Hom_{\cO_{\cF/H_{r-2}}}(\cEi_{G/H_{r-2}}\cO_{\cF_{H_{r-2}}}, 
V(\cF/H_{r-2})) \lra \cdots \\ 
\lra
\Hom_{\cO_{\cF/H_{1}}}(\cEi_{G/H_{1}}\cO_{\cF/H_1}, 
V(\cF/H_{1})) \lra 
\Hom_{\cO_{\cF}}(\cEi_{G}\cOcF, 
V(\cF/1)) 
\end{multline*}
\end{defn}

\begin{remark}
  Even if $V(\cF/H)$ is $\cF/H$-torsion, it does not follow that
  $\Hom_{\cOcFH}(\cEi_{G/H}\cOcFH, V(\cF/H))$ is $\cF/H$-torsion. 
  \end{remark}

\subsection{Right adjoints to evaluation}
From an algebraic point of view the simplest behaviour comes from
objects that let us calculate in terms of rings rather than diagrams
of rings. For this we use  right adjoints to evaluation at a subgroup.

The idea is that the adjoint forms are constant
above a given point. This is correct if the subgroup in question is of
codimension $\leq 1$, but not in general, because of the need to use
torsion objects and the distinction between sums and
products.

\begin{example} If $G$ is of rank 2, the 
functors representing evaluation are as follows.
$$\at_G(V)=f_G(V)=\left[ \cEi_G \cOcF \tensor V\lra 0 \lra 0\right]$$
$$\at_H(T(H))=\left[ \cEi_G\cOcF\tensor 
\Hom_{\cOcFH}(\cEi_{G/H}\cOcFH, T(H))\lra 
\cEi_H\cOcF\tensor_{\cOcFH} T(H) \lra 0\right] $$
$$\at_1(S)=\left[ \cEi_G\cOcF\tensor 
\Gamma_{\Sigma}\Hom_{\cOcF} (\cEi_G\cOcF, S)\stackrel{h^0}\lra 
\bigoplus_H \cEi_H\cOcF\tensor_{\cOcFH}\Gamma_{\cF /H}\Hom_{\cOcF}
(\cEi_{H}\cOcF, S)
\stackrel{h^1}\lra
S\right]$$
where $\Gamma_{\cF/H}$ indicates torsion $\cOcFH$-modules and $\Gamma_{\Sigma}$ refers to the elements mapping into the
the torsion submodules, and into the sum rather than the
product. These two phenomena occur in rank 2 for the first time. 

\end{example}

Equipped with this example we can define the objects $\at_L(T)$.

\begin{defn}
\label{defn:atL}
Given a torsion $\cOcFL$-module $T$, we may define an object
$\at_L(T)$ by taking its torsion components to be 
$$V(\cF /H)=\dichotomy
{\Gamma_\Sigma \Gamma_{\cF/H} \Hom_{\cOcFL}(\cEi_{H/L}\cOcFL , T)
 & \mbox{ if } K\supseteq L } 
{0 & \mbox{ if } K\not\supseteq L } 
$$
The functor $\Gamma_{\cF/H}$ takes the torsion submodule and
$\Gamma_\Sigma$ takes the submodule mapping into the sum; we will
describe them fully in the next subsection and prove these properties
(Lemma \ref{lem:Gamma}), but for the present we
need only know they give natural submodules of the Hom functor.
\end{defn}



\begin{lemma}
The functor $\at_L$ is right adjoint to evaluation at $L$: for a
torsion $\cOcFL$-module $T$, we have 
$$\HomAt (X, \at_L(T))=\Hom_{\cOcFL} (V_X(\cF/L), T)$$
\end{lemma}

\begin{proof}
We describe the counit and unit of the adjunction. The triangular
identities follow from those of the Hom-Tensor adjunction. 

The counit $ev_L \at_L(T)\lra T$ is the identity. For the unit
$$X\lra \at_L(V_X(\cF/L)),$$ 
let us suppose that $L$ is of codimension $t$ and that $s\geq
t$. In the display, $H$ runs through subgroups of codimension $s+1$
containing $L$ and $K$ runs through subgroups of $H$ containing $L$
with  codimension $s$ in $G$.
$$\xymatrix{
\bigoplus_{H}  \at_L(T)(\cF/H) 
\ar@{=}[d] \rto  &
\bigoplus_{K}  \at_L(T)(\cF/K) 
\ar@{=}[d] \\
\bigoplus_{H\geq L} \cEi_{H/K}\cOcFL\tensor_{\cOcFH} \Gamma_{\Sigma}\Gamma_{\cF/H} \Hom_{\cOcFL}(\cEi_{H/L} \cOcFL, T) 
\rto  &
\bigoplus_{K\geq L}  \Gamma_{\Sigma}\Gamma_{\cF/L} \Hom_{\cOcFL}(\cEi_{K/L} \cOcFL, T) 
}$$
\end{proof}

\subsection{The torsion subfunctor}
Certain limit constructions do not preserve torsion objects, or do not
preserve the property of mapping into the sum of components, so it is
useful to formalize  a bigger category in which the constructions can
be made, together with a right adjoint for returning to $\cAt(G)$.

\begin{defn}
 The category $\cAhatt(G)$, has objects $X$ consisting 
collections $\{ V(\cF/K)\}_K$ with $V(\cF/K)$ an $\cOcFK$-module with 
structure maps
$$h^K_L: \cEi_{K/L}\cOcFL \tensor_{\cOcFK} V(\cF /K)\lra V(\cF/L).$$
\end{defn}

Taking components gives an obvious inclusion functor
$$i: \cAt(G)\lra \cAhatt(G). $$

\begin{lemma}
  \label{lem:Gamma}
The inclusion $i$ has a right adjoint $\Gamma: \cAhatt(G)\lra \cAt(G)$. 
\end{lemma}
\begin{proof}
  The functor $\Gamma$ is the composite $\Gamma=\Gamma''\Gamma'$.
  The functor $\Gamma'$ takes the torsion submodule  $\Gamma_{\cF/K}
  V(\cF/K)$ at $K$. This is compatible with the structure maps, since any
  $\cF/K$-torsion element is $\cF/L$-torsion for $L\leq K$. 

The functor $\Gamma''$ is defined by taking $\Gamma_\Sigma$ at each
point, where $\Gamma_\Sigma$  can be defined by induction on the
dimension of $H$ using the pullback square 
 $$\xymatrix{
 \Gamma_{\Sigma}\Gamma_{\cF/H}
 \Hom_{\cOcFL}(\cEi_{H/L}\cOcFL , T)\dto \rto &
 \Hom_{\cOcFK}(\cEi_{H/K}\cOcFK, 
 \bigoplus_{K\leq H}
 \Gamma_\Sigma \Gamma_{\cF/K}\Hom_{\cOcFL}(\cEi_{K/L}\cOcFL , T))\dto \\
 \Hom_{\cOcFL}(\cEi_{H/L}\cOcFL , T) \rto &
 \Hom_{\cOcFK}(\cEi_{H/K}\cOcFK,  \prod_{K\leq H}
 \Hom_{\cOcFL}(\cEi_{K/L}\cOcFL , T)) 
 }$$
\end{proof}

In particular this lemma formalizes Definition \ref{defn:atL}.

\section{Injectives}
\label{sec:inj}
We need to show there are enough injectives and describe them in
a way that allows us to do computations. Since $\at_K$
is a right adjoint, it is obvious that, for any torsion injective
$\cOcFK$-module $I$, the object
$\at_K(I)$ is injective in the torsion model.  Indeed, this is the
main reason for introducing the $\at_K$ construction. 

Next, we know that there are enough injective torsion
$\cOcFK$-modules, which can be constructed as sums of those of the form 
$H_*(BG/\tK)$, where $\tK$ runs through subgroups with identity 
component $K$.

Finally it is both illuminating and convenient to give a more explicit
description of the functors $\at_K(H_*(BG/\tK))$. After some
recollections about Gorenstein rings in Subsections \ref{subsec:GorD}
and \ref{subsec:GorDllp},  we give the description in Subsection \ref{subsec:atK}. 

\subsection{Enough injectives}
\label{subsec:sumsinj}
We wish to  construct enough injectives by taking products of those with
explicit constructions. 

\begin{cor}
\label{cor:prods}
The category $\cAt(G)$ has products. 
\end{cor}

\begin{proof}
Given objects $X_i$ of $\cAt(G)$ we wish to say that the $K$-torsion 
component of the product $\prod_i X_i$ is $\prod_i 
V_i(\cF/K)$. This lies in $\cAhatt(G)$, so by Lemma \ref{lem:Gamma} we
may apply the right adjoint 
$\Gamma$ to obtain an object of $\cAt(G). $
\end{proof}

\begin{lemma}
\label{lem:enoughinj}
There are enough injectives which are products of 
injectives of the form $a_K(I)$ for torsion  injective $\cOcFK$-modules $I$. 
\end{lemma}

\begin{remark}
Exactly as in the rank 1 semifree case (Remark \ref{rem:suminjnotinj}),  an arbitrary sum of
injectives need not be injective. 
\end{remark}

\begin{proof}
For each subgroup $\Kt$, there are enough torsion injective
$H^*(BG/\Kt)$-modules formed as direct sums of
$H_*(BG/\Kt)$. This means that for any $X$ there is a monomorphism
$V(\cF/K) \lra  I(\cF/K))$ where $I(\cF/K)$ is a sum of
$\cOcFK$-modules $H_*(BG/\Kt)$. There is a corresponding map 
$X \lra  \at_K(I(\cF/K))$ monomorphic at $K$, and hence a monomorphism
$$X\lra \prod_k \at_K(I(\cF/K)).$$
\end{proof}

\subsection{Gorenstein duality}
\label{subsec:GorD}
One particular case of the $\at_K$ construction is especially 
important: the one supplying enough injectives. Indeed, we know that 
there are enough injective torsion $\cOcFK$-modules of the form 
$H_*(BG/\tK)$, where $\tK$ runs through subgroups with identity 
component $K$. It is illuminating to have a description of the entries 
in $\at_K(H_*(BG/\tK))$ in simple terms. We are able to do this using 
the fact that polynomial rings are Gorenstein. The basic idea is very 
simple, but there are some issues to highlight along the way.


The ring $H^*(BG)$ is Gorenstein and 
$$H^*_{\fm}(H^*(BG))=H^r_{\fm}(H^*(BG))\cong \Sigma^r\Hom_k(H^*(BG), 
k)=\Sigma^r H_*(BG). $$
There are two notable things about this. First, 
the isomorphism is not natural in the sense that it should be twisted
by the determinant in order to be natural for ring isomorphisms. Second, the $\Hom_k$ consists of {\em graded} maps. If
we used all maps we would obtain a completion of $H_*(BG)$, and we
would need to pass to cellularizations to recover $H_*(BG)$ itself. 

We need something  more general. For a Gorenstein local
ring $R$ of dimension $r$ we have 
$$H^*_{\fm}(R)=H^r_{\fm}(R)\cong E_R(k)$$
where $E_R(k)$ is the injective hull of the residue field $k$. If $R$
happens to be a $k$-algebra, we may form $\Ehat_R(k)=\Hom_k(R,k)$, and
the map $k=\Hom_k(k,k)\lra \Hom_k(R,k)=\Ehat_R(k)$ extends to an
embedding $E_R(k)\subseteq \Ehat_R(k)$. However we note that if $R$ is
of countably infinite dimension as a $k$-vector space $\Ehat_R(k)$ will be of
uncountable dimension, whilst $E_R(k)$ will be of countable
dimension, so $E_R(k)\neq \Ehat_R(k)$ in general. 

Accordingly the best we can hope for is that we have a monomorphism 
$$H^r_{\fm}(R)\lra \Hom_k(R,k)=\Ehat_R(k)$$
with the image being a copy of $E_R(k)$. When $R$ is polynomial, we can
recover this since we have a residue map 
$$\res: H^r_{\fm}(R)\tensor \Omega^r_{R/k}\lra k, $$
where $\Omega^r_{R/k}$ is the module of K\"ahler differentials, with
$\Omega^r_{R/k}\cong R$.  This 
gives rise to 
$$H^*_{\fm}(R)=H^r_{\fm}(R)\lra \Hom_k(\Omega^r_{R/k}, 
k)\cong \Ehat_R(k) $$

\subsection{Gorenstein duality for linear-localized polynomial rings}
\label{subsec:GorDllp}

Our situation has all the features described in Subsection
\ref{subsec:GorD}
but is not quite standard since
the rings are not local. For a connected subgroup $K$ of codimension $s$, we need to consider the ring
$R=\cEi_KH^*(BG)$. 

We think of $H^*(BG)$ as polynomial functions on the affine space
$TG=\spec(H^*(BG))$. Corresponding to the short exact sequence 
$$1\lra K \lra G \lra G/K\lra 1$$ 
we have maps 
$$H^*(BK)\lla H^*(BG)\lla H^*(BG/K)$$
and a fibration 
$$TK \lra TG \lra TG/K. $$
The ring $H^*(BG/K)$ is  thus naturally a
subring of $H^*(BG)$. Choosing particular degree $-2$ generators
$x_i$, we have $H^*(BG/K)=k[x_1, x_2, \ldots , x_s]$ and the kernel of
restriction to $K$ is the ideal $\fm_{G/K}=(x_1, \ldots ,x_s)$. 

To go further we choose a splitting.  If $H^*(BK) = k[y_1, y_2, \ldots
, y_t]$, we have $H^*(BG)\cong k[x_1, \ldots , x_s, y_1, \ldots
, y_t]$.   The linear forms in $\cE_K$ are  precisely those forms
involving some $y_i$ (or, intrinsically, those not in the image of 
$H^*(BG/K)$).  We therefore have a natural map $\cEi_KH^*(BK)\lla
\cEi_K H^*(BG)$  which 
plays the role of the map from $R$ to its residue field. 

Since $K$ is connected, we may choose a splitting of the inclusion  $K
\lra G$ and hence make  $R=\cEi_KH^*(BG)$  an algebra over 
$\vark = \cEi_KH^*(BK)$. 

The ring  $R$ consists of meromorphic functions on $T G$ regular on 
$TK$ but with denominators that are products of linear
forms. 

\begin{lemma}
\label{lem:GorD}
We have a natural embedding
$$H^*_{\fm_{G/K}}(\cEi_K H^*(BG))=H^{s}_{\fm_{G/K}}(\cEi_K H^*(BG))\lra 
\Sigma^{2s}\Hom_k(\cEi_K H^*(BG), 
k),  $$
giving an isomorphism
$$H^{s}_{\fm_{G/K}}(\cEi_K H^*(BG))\lra 
\Sigma^{2s}\Gamma_{\cF /K}\Hom_{\vark} (\cEi_K H^*(BG), 
\vark ),  $$
\end{lemma}

\begin{proof}
Choose a splitting  $H^*(BG)\cong H^*(BG/K)\tensor
H^*(BK)$, and note that 
$$H^*_{\fm_{G/K}}(H^*(BG))=H^*_{\fm_{G/K}}(H^*(BG/K))\tensor H^*(BK). $$
As in Subsection \ref{subsec:GorD}
$H^*_{\fm_{G/K}} (H^*(BG/K)) =H^{s}_{\fm_{G/K}} (H^*(BG/K))$
has a basis consisting of monomials $x_1^{i_1}x_2^{i_2}\cdots
x_{s}^{i_{s}}$ with all exponents negative. Thus the top degree basis element is
$(x_1x_2\cdots x_s)^{-1}$  and  the map  to $\Hom_k(H^*(BG/K), k)$ 
is given by residues. This gives an isomorphism
$$H^*_{\fm_{G/K}}(H^*(BG/K))\cong \Gamma_{\fm_{G/K}}\Hom_k(H^*(BG/K), 
k). $$
With our usual convention of taking graded Hom, the
$\Gamma_{\fm_{G/K}}$ could be omitted since $H^*(BG/K)$ is finite
dimensional in each degree. 
Now tensor this with $H^*(BK)$ and localize:
\begin{multline*}
\cEi_K H^*_{\fm_{G/K}}(H^*(BG))=
\cEi_K \left [ H^*(BK)\tensor H^*_{\fm_{G/K}}(H^*(BG/K))\right] \cong \\
\cEi_K \left[ H^*(BK)\tensor \Gamma_{\fm_{G/K}}\Hom_k(H^*(BG/K), 
  k)\right]
\end{multline*}

We continue to write $\cE_K\subseteq H^*(BG)$ for the multiplicatively
closed set of $K$-essential representations of $G$. Having chosen a
splitting of $K\lra G$, we need a different notation for those arising
from representations of $K$, so write $\cE'_K=\cE_K\cap H^*(BK)$.
Now if $M$ is an $\fm_{G/K}$-power torsion module then  inverting
$\cE_K$ is the same as inverting $\cE'_K$ (indeed, if
$u$ is invertible and $x\in \fm_{G/K}$ then $u+x$ is invertible).
Accordingly, $\cEi_K H^*_{\fm_{G/K}}(H^*(BG))$ is a free module over
$k_K=\cEi_K H^*(BK)$ on the monomial basis of negative powers above.
Thus $\cEi_K H^*_{\fm_{G/K}}(H^*(BG))$ embeds in
$\Hom_{k_K}(\cEi_K H^*(BG), k_K) $.

Now consider the $\Gamma_{\cF /K}$-torsion in a module $M$. It has a
filtration
$$0\subseteq \ann (\fm_{G/K}, M)\subseteq \ann (\fm_{G/K}^2,
M)\subseteq \ann (\fm_{G/K}^3, M) \subseteq \cdots   \subseteq \Gamma_{\cF/K}M.$$
The subquotients are  modules over $\cEi_K H^*(BG)/\fm_{G/K}=\cEi_K
H^*(BK)$.

Taking $M=\Hom_{k_K}(\cEi_KH^*(BG), k_K)$ we prove by induction on $a$
that $\ann (\fm_{G/K}^{a+1}, M)/\ann (\fm_{G/K}^{a}, M)$ is a free
module on the negative monomials of total degree $-s-a$.  The
monomials are independent, so it remains to show they span. We know
$\cEi_K H^*(BG)$ is free over $k_K$ on the monomials, so we need only
observe that a function annihilated by $\fm_{G/K}^a$ is zero on
monomials of degree $\geq a+1$. 

\end{proof}

\subsection{Explicit description of generating injectives}
\label{subsec:atK}

Using the work of the last two subsections we can give an 
attractive description of $\at_L(H_*(BG/\Lt))$ showing that it
captures global geometry. 

\begin{lemma}
\label{lem:injasloccoh}
For a subgroup $\Kt$, the  torsion part of  $\at_L(H_*(BG/\Lt))$ is given by 
$$V(G/\Kt)=\dichotomy
{\Sigma^{-2\dim(G/K)}H^{\dim(G/K)}_{\fm_{G/\Kt}}(\cEi_{K/L}H^*(BG/\Lt))  & \mbox{ if
    $\Lt$ is cotoral in $\Kt$ }} 
{ 0 & \mbox{ otherwise }} $$
\end{lemma}

\begin{proof}
The vanishing is clear. Let $s=\dim (G/K)$. By Lemma \ref{lem:GorD},
we have an isomorphism 
$$H^{s}_{\fm_{G/K}}(\cEi_K H^*(BG/\Lt))\stackrel{\cong}\lra 
\Sigma^{2s}\Gamma_{\cF/K}\Hom_{\vark} (\cEi_{K/L} H^*(BG/\Lt), 
\vark ).  $$
The result follows form  Definition \ref{defn:atL}, once we take
account of the effect of    $\Gamma_\Sigma$. 

First, we observe that the embedding maps into $\Gamma_\Sigma$.
We have already seen that 
an element $h=\lambda/e(V)\tensor f/x_1^{i_1} x_2^{i_2}\cdots x_s^{i_s}$
automatically maps to a torsion element. It remains to observe that it is only nonzero
for finitely many $L$. Indeed, if $L$ is of codimension 1 in  $K$ and
$\Lt$ is cotoral in $\Kt$ then  $H^*(BG/\Lt)=k[x_1, \ldots , x_s, z]$
for some independent linear form $z$. Then $h$ will only map to a
nonzero element if it is not regular, and this only happens if $z$
occurs amongst the finitely many linear factors of  $e(V)$.
Since the map is surjective by Lemma \ref{lem:GorD} it follows that
$\Gamma_\Sigma$ is the identity on this object. 
\end{proof}

The structure maps are given by the relative residue maps. 

\begin{lemma}
For cotoral inclusions $\Ht\supseteq \Kt \supseteq \Lt$ the structure
map 
\begin{multline*}
h^{\Ht}_{\Kt}: \cEi_{H/L}H^*(BG/\Lt )\tensor
\Sigma^{-2\dim(G/H)}H^{\dim (G/H)}(\cEi_{H/L}H^*(BG/\Lt))\\
\lra 
\Sigma^{-2\dim(G/K)}H^{\dim(G/K)}(\cEi_{K/L}H^*(BG/\Lt))
\end{multline*}
is given by the relative residue.\qqed 
\end{lemma}



 


\section{Injective dimension}
\label{sec:id}

For the {\em convergence} of the Adams spectral sequence it is 
enough to know $\cAt(G)$ has finite injective dimension. This is an
easy consequence of the finite injective dimension of torsion modules
over a polynomial ring, and the proof is given in Proposition \ref{prop:id}.
 For the {\em  use} of the Adams Spectral Sequence  it is enough to be able to  
calculate with $\cAt(G)$, and for this purpose one does not need to
know the exact injective dimension either,  so it will suffice to 
give a finite upper bound.


\subsection{An upper bound for the injective dimension}
It is very easy to give an upper bound, simply using the fact that the
injectives $\at_L(I)$ are only non-zero at $K$ when $K$ contains
$L$. Indeed, we may argue by induction on the codimension of support
that an object $X$ with support in codimension $\leq c$ is of finite
injective dimension. This is obvious for $c=0$ since $\at_G(V)$ is
always injective. Supposing  that $X$ has support in codimension $c$
and that objects with support of lower codimension are of finite
injective dimension then we may
construct the start of an injective resolution by focusing on
subgroups $L$ of codimension $c$ and starting an injective resolution
$$0\lra X \lra \bbI_0\lra \bbI_1\lra \cdots \lra  \bbI_c$$
so that for each subgroup $L$, the torsion components $V(\cF/L)$ give
an injective resolution of $V_X(\cF/L)$. The cokernel of
$\bbI_{c-1}\lra \bbI_c$ has support in codimension $\leq c-1$ and it
is therefore of finite injective dimension by induction. This gives a
bound quadratic in $c$. With an additional fact we can get a
linear bound.

\begin{lemma}
  \label  {lem:atIcomponents}
If $I$ is an injective torsion $\cOcFK$-module then the components of 
$\at_K(I)$ are all injective. 
\end{lemma}

\begin{proof}
  We have seen the component at $H$ is zero unless $K\leq H$ and 
  the component at $H$ is 
  $$\Gamma_\Sigma \Gamma_{\cF/H}\Hom_{\cOcFK}(\cEi_{H/K}\cOcFK, I)).$$
  This is formed from the injective $I$ by applying three right 
  adjoints, $  \Hom_{\cOcFK}(\cEi_{H/K}\cOcFK, \cdot )), $
  $\Gamma_{\cF/K}$ and $\Gamma_\Sigma$ so the result is also 
  injective. 
  \end{proof}

  \begin{prop}
    \label{prop:id}
An object $X$ of codimension $c$ has injective dimension $\leq 2c$. In
particular, the injective dimension of $\cAt(G)$ is $\leq 2r$.
\end{prop}

\begin{proof} We will give the argument for a general object (i.e.,
  with support in codimension $\leq r$), but clearly it applies to any
  codimension.
  
  For an object $X$ of $\cAt(G)$, we describe how to construct an injective resolution
$$0\lra X \lra \bbI_0\lra \bbI_1\lra \cdots \lra \bbI_{2r}\lra 0$$
of length $2r$, and we write $X=X_0$ and for $i\geq 1 $
we take $X_i=\cok (X_{i-1} \lra \bbI_{i-1})$.

There are two halves to the construction
\begin{itemize}
\item For   $X_0, \ldots , X_r$ the injective dimensions of the individual
  components is steadily reduced: noting that if $K$ is of codimension
  $c$ then the injective dimension of torsion $\cOcFK$-modules is $c$,
  we ensure that $X_s(\cF/K)$ is either  injective or of injective
  dimension $\leq c-s$. Thus the torsion components of $X_r$ are
  all injective.
\item For $X_{r}, X_{r+1}, \ldots , X_{2r} $ we retain the
  property that all torsion components are injective but we ensure
  steadily more of them are zero, so that $X_{r+i}(\cF/K)=0$ for
  $\dim (K)<i$. 
\end{itemize}

This is rather straightforward. We construct the resolution
recursively. We take $X=X_0$ and for $s\geq 0$ we suppose we
have constructed up to $X_s$ in the exact sequence
  $$0\lra X \lra \bbI_0\lra \cdots \lra \bbI_{s-1}\lra X_s\lra 0 $$
For each $K$ we choose an
injective torsion $\cOcFK$-module $I_s(K)$ and a monomorphism $i_s(K): X_s (\cF/K)\lra
I_s(K)$. If $X_s(\cF/K)$ is already injective, we take
$I_s(K)=X_s(\cF/K)$ and $i_s(K)$ to be the identity.
Now take 
$\bbI_s=\prod_K \at_K(I_s(K))$ and define
$i_s : X_s\lra \bbI_s$ by ensuring the $\at_K(I_s(K))$ component
corresponds to $i_s(K)$ under the adjunction
$$\Hom_{\cAt(G)} (X_s, \at_K(I_s(K)))=\Hom_{\cOcFK} (X_s(\cF/K), I_s(K)).$$
This ensures the map $i_s$  is a monomorphism and we take
$X_{s+1}=\cok (i_s: X_s\lra \bbI_s)$.

By Lemma \ref{lem:atIcomponents}, the value of $\bbI_s$ at $\cF/K$ is
injective, so that if $X_s(\cF/K)$ is not already injective, the
injective dimensions are related by 
$\id (X_{s+1}(\cF/K)) =\id (X_s(\cF/K))-1$. This deals with the first
half of the construction. For the second half we suppose that  that
$s=r+i$ with $i\geq 0$ and $X_r, X_{r+1}, \ldots, X_{r+i}$ have support as required,
so that in particular $X_{r+i}$ has support in dimension $\geq i$.
By construction $\bbI_{r+i}$ also has support in dimension $\geq i$,
and if $K$ is of dimension $i$,  the only factor of  $\bbI_{r+i}$ not zero at $K$ is
$\at_K(I_{r+i}(K))$, so that
$$X_{r+i+1}(\cF/K)=\cok \left[X_{r+i}(\cF/K)\lra
  \bbI_{r+i}(\cF/K)=I_{r+i}(K)
\right]=0 $$
as required. We note that when $i=r$ the conclusion is $X_{2r+1}=0$.
\end{proof}

 \subsection{Local duality}
 \label{subsec:localduality}
The purpose of this section is to observe that if $T$ is an Artinian
torsion module then $a_L(T)$ is of injective dimension $\leq \dim (G/L)$ as
one might expect. The proof is straightforward using local duality.

One might view this as saying that most of the objects of $\cAt(G)$
that we need to consider have injective dimension $\leq r$. However
this is a bit misleading, since even simple operations like infinite sums may give objects of
higher injective dimension. 

\begin{prop}
  \label{prop:artid}
  If $T$ is an Artinian $\cOcFL$-module then $\id_{\cOcFL} (T)=\id_{\cAt(G)} (\at_L(T))$.
 \end{prop}

  \begin{proof}
    Since $T$ is Artinian, we may choose 
  a resolution by Artinian injective torsion modules 
  $$0\lra T \lra I_0\lra I_1\lra \cdots \lra I_d\lra 0. $$
  We will show that 
  $$0\lra \at_L(T) \lra \at_L(I_0)\lra \at_L(I_1)\lra \cdots \lra \at_L(I_d)\lra 0$$
  is exact, which gives the desired conclusion.

  The point is that if we take $I=H_*(BG/L)$ then by Lemma
  \ref{lem:injasloccoh} the torsion
  component of $\at_L(I)$  at $\cF/K$ is
  $H^t_{\fm_{G/K}}(\cEi_KH^*(BG/L))$, so we need to see that this
  process is exact.

  For this, we apply local duality. The idea is that the
injective $\cOcFK$-module resolution $I_\bullet$ can be recognized as the local
cohomology of the dual of a free $\cOcFK$-module resolution $F_\bullet$ of a dual
module $N$. By a similar process after localization, the values at
other levels are then recognized as local
cohomology of localizations of $N$. The exactness of
$I_\bullet$ then gives exactness of these. 

The key here is that for the polynomial ring $P=k[x_1, \ldots , x_s]$ we have
$$H^*_\fm(P)=H^s_\fm(P)=\Sigma^{2s} \Hom_k(P,k). $$
As usual the $\Hom_k$ refers to graded maps, and if we used ungraded
maps we would insert $\Gamma_\fm$ to achieve the same end. From the
variance we can see this cannot be natural, but we can easily correct
that: for arbitrary free $P$-modules $F$ we have a natural isomorphism 
$$\Hom_k( H^s_\fm (F), k)=\Sigma^{-2s} \Hom_P  (F, P). $$
Thus if $F_\bullet$ is a free resolution of a
$P$-module $N$ we have
$$\Hom_k( H^{s-i}_\fm (N), k)=\Sigma^{-2s} \Ext^{i}_P  (N, P) $$

Note that we have natural maps 
$$H^s_\fm (F)\lra \Sigma^{2s}\Hom_k (\Hom_k  (H^s_\fm(F), k),k) \mbox{ and }
F\lra \Sigma^{2s}\Hom_P (\Hom_P (F,P), P), $$
both of  which are  isomorphisms if $F$ is of finite rank. Thus if
$M$ is Artinian, the injective resolution $I_\bullet$ of $M$
determines a free resolution $F_\bullet$ of $N=\Hom_k(M,k)$.
Similarly $\cEi_K F_\bullet$ is a free resolution of
$\cEi_K\Hom_k(M,k)$, and then by taking
$\Hom_{k_K}(\cdot  , k_K)$ we obtain an
injective resolution of $\Hom_{k_K}(\cEi_K M, k_K)$.

By local duality this is 
$$\at_L(I_{\bullet}(G/K))=\Sigma^{-2t}H^{t}_{\fm_{G/K}}(\Hom_{\cEi_K
  P}(\cEi_KF_\bullet, k_K)) , $$
showing that the complex is exact as required. 
\end{proof}

\subsection{Attainment of the bound}
\label{subsec:attained}
The aim of this subsection (not currently achieved!) is to establish
the exact injective dimension of $\cAt(G)$ by writing down an object
of injective dimension equal to the upper bound established in
Proposition \ref{prop:id}. We showed in Proposition \ref{prop:artid}
that if $T$ is Artinian  $\at_L(T)$ has the same injective dimension as $T$ does, and
hence all objects of this type have injective dimension $\leq r$.

We will show that for $r\geq 1$ there are objects $X$ with codimension
$1$ that are of injective dimension 2. This only establishes the injective
dimension of $\cAt(G)$ for the circle group. The same argument
shows that there is an object of codimension $c$ with injective
dimension $c+1$ in general, so that we only know that the injective
dimension of $\cAt(G)$ lies between $r+1$ and $2r$.

\begin{lemma}
  If $X$ has support in codimension 1 then
  \begin{itemize}
    \item $X$ is injective if and only if $V(\cF/G)\lra \Hom_{\cOcFH}( \cEi_{G/H}
      \cOcFH, V(\cF/H))$ is surjective for all $H$ of codimension 1.
    \item $X$ is of injective dimension $\leq 1$ if and only if
      $$\Ext_{\cOcFH}^1(\cEi_{G/H}\cOcFH, V(\cF/H))= 0 \mbox{ for all subgroups
        $H$ of codimension 1}$$
      \end{itemize}
  \end{lemma}

\begin{lemma}
\label{lem:aspherical}
$$\Ext^s_{\cOcFH}(\cEi_H\cOcF , T)=
\trichotomy{\Hom_{\cOcFH}(\cEi_H\cOcF , T) \mbox{ if } s=0}
{\ilim^1(T, \cE_H) \mbox{ if } s=1}
{0 \mbox{ if } s\geq 2}$$
\end{lemma}

For a commutative ring $R$ and a
multiplicatively closed set $\cE$, we have 
$$\Ext_R^1(\cEi R, M)=\ilim^1 (M,\cE). $$
We make some observations about vanishing and non-vanishing of this. 

\begin{lemma}
Suppose $R$ is a commutative ring and $\cE$ is a multiplicatively
closed subset. There is an isomorphism
$$\ilim^1_{e\in \cE} (eM)\cong M_{\cE}^\wedge/M. $$
There is an epimorphism 
$$\ilim^1_{e\in \cE} (M,e)\lra R^1\ilim_{e\in \cE} (eM)$$ 
\end{lemma}

\begin{proof}
For the first statement we have the inverse system with $e,f\in \cE$
$$\xymatrix{0\rto &eM \rto & M\rto &M/eM \rto &0\\
0\rto &efM \uto \rto & M\uto^{=} \rto &M/efM \uto \rto &0 
}$$
where the maps are inclusions and projections. 
Taking inverse limits gives a six term exact sequence with the last
two zero. 

For the second we have a inverse system 
$$\xymatrix{
0\rto &\ann_M (e) \rto & M\rto^e  &eM \rto  &0 \\
0 \rto &\ann_M(ef)  \rto \uto^f& M\rto^{ef}\uto^f &efM \uto \rto &0 
}$$
The last map in the six term exact sequence is an epimorphism. 
 \end{proof}

In particular if we may take $R=k[x_1, \ldots , x_r]$ to be a polynomial
ring on generators of degree $-2$. Now take $M=\bigoplus_{n\geq 0}
\Sigma^{2n}R/\fm^n$ 
and $\cE=\langle x_1\rangle$, and we see that  $M$ is not 
$\cE$-complete  and hence $R^1\ilim(M,\cE)\neq0$. 

This is enough to construct an object of $\cAt(G)$ of injective
dimension $r+1$ when $r\geq 1$. Indeed we may take 
$M=\bigoplus_n \Sigma^{2n}k[c_1]/c_1^n$. We then find 
$$\Gamma_{\cF/L}\Ext^1(\cEi_L\cOcF, M)\neq 0.  $$
Since the module is annihilated by $c_2, \ldots , c_r$ it has
injective dimension $r-1$. Hence $\at_1(M)$ has injective dimension
$\geq 2+(r-1)=r+1$.





\section{Ind-corepresenting evaluation}
\label{sec:indcorep}
\newcommand{\Kts}{\tilde{K}^*}
Given an object $X$ of $\cAt (G)$, we would like to be able to 
determine its torsion modules $V(\cF/K)$ by considering maps from 
objects of $\cAt(G)$, in the same way that homotopy groups are given 
by maps out of a sphere. Although this is not possible as it stands, 
we can recover $V(\cF/K)$ as a colimit of such values. Since
$V(\cF/K)=\bigoplus_{\Kt}V(G/\Kt)$ we focus on a single subgroup $\Kt$.

This phenomenon should be rather familiar. 
Given a torsion $k[c]$-module $T$ we may recover $T$ as 
$T=\Hom_{k[c]}(k[c], T)$, but if we are restricted to using torsion 
modules in the domain this is not available to us. However, 
$\ann (c^n,T)=\Hom_{k[c]}(k[c]/c^n, T)$, so 
$$T=\bigcup_n \ann (c^n,T)=\colim_n \Hom_{k[c]}(k[c]/c^n, T).$$
In this sense the module is ind-corepresented by the inverse system $\{ k[c]/c^n\}_n$. 
The situation with $\cAt (G)$ is a little more complicated, so we 
start with the circle group in the semifree case as in Section \ref{sec:semifreecircle}. 

\begin{example} {\em (The circle group)}
By the discussion above we see immediately that evaluation at $G/1$ is ind-corepresented by $\{ (0\lra 
k[c]/c^n)\}_n$. This is an object familiar from topology since 
$\pi^G_*(DS(nz)_+)=k[c]/c^n$, where $D$ denotes Spanier-Whitehead
duality, so that  $\piAts(DS(nz)_+)=(0\lra 
k[c]/c^n)$.  In fact the representing inverse system is the homotopy of 
the dual of the  direct system $S(z)_+\subset S(2z)_+\subset S(3z)_+\subset \cdots$. 

Evaluation at $G/1$ is ind-corepresented by  the inverse system $\{ t\tensor k \lra 
k[c,c^{-1}]/c^nk[c])\}_n$. To define a map 
$$e: \colim_n \Hom \left( [t\tensor k\lra k[c,c^{-1}]/c^nk[c]],  [t\tensor  
V\lra T]\right) \lra V$$ 
we proceed as follows.  Consider a map $\theta$ from the $n$th term
$$
\xymatrix{
t\tensor k \dto \rto^{1\tensor \theta(G/G)}&t\tensor V\dto^q\\
 k[c,c^{-1}]/c^n k[c]\rto^(.7){\theta(G/1)}&T
}$$
and note first that it is determined by  $\theta (G/G)$. This is
because the vertical structure
map in the domain is an epimorphism. Thus $\theta $ is determined by
 evaluation  at $1\tensor \iota $. This is compatible with the maps in
the inverse system as $n$ varies (since all are the identity at $G/G$), and so we may
define $e([\theta ])=\theta (G/G) ( \iota)$ and obtain a well defined
injective map $e$. Every element in the image of $q$
is divisible, so that if $q(1\tensor v)$ is annihilated by $c^s$ there 
is a map $\theta$ with evaluation $v$  if $n\geq s$, and since $T$ is
torsion the map $e$ is surjective. 

Again the objects are geometrically familiar since  
 $\piAts (S^{-nz})=(t\tensor \Q \lra k[c,c^{-1}]/c^nk[c])$. 
 In fact the representing inverse system is the homotopy of 
the dual of the system $S^0\subset  S^z \subset S^{2z} \subset S^{3z}\subset \cdots$
\end{example}

In principle we could continue purely algebraically, but the
topological motivation will lead us efficiently to a solution. 
Recall that 
\begin{multline*}
\piAts (X)(G/\Kt)=\piGs(\siftyV{K}\sm EG/\Kt_+\sm X)\\
=[S^0, \colim_{V^K=0}S^V \sm \colim_n EG/\Kt^{(n)}_+\sm X]^G 
=\colim_{V^K=0,n}[S^{-V} \sm DEG/\Kt^{(n)}_+, X]^G 
\end{multline*}
This suggests that if we choose a nice  filtration of the 
universal space (indicated by superscript $(n)$), we could use 
$$B_{\Kt} (V,n)=\piAts  (S^{-V}\sm DEG/\Kt_+^{(n)})$$
to give ind-corepresenting objects.

We choose to be very explicit so as to make this more easily
digested, but pay the price of making unnecessary choices. We are
ind-corepresenting evaluation at  $G/\Kt$, so $\Kt$ is fixed.
The corepresenting object will be zero except at subgroups cotorally
beneath $\Kt$, so we choose a subgroup $\Lt$ so that $\Kt/\Lt$ is a
torus.

We have a short exact sequence
$$1\lra \Kt/\Lt \lra G/\Lt \lra G/\Kt\lra 1$$
giving a short exact sequence of commutative $k$-algebras  
$$k \lla H^*(B\Kt/\Lt) \lla H^*(BG/\Lt) \lla H^*(BG/\Kt)\lla k. $$
This is natural, but because   $\Kt /\Lt$ is a torus we may go further
and choose a splitting.  
Thus the projection $G/\Lt\lra G/\Kt$ is split, and we may choose $\Kts\supseteq \Lt$ so that the composite $\Kt /\Lt \lra G/\Lt\lra G/\Kts$
is an isomorphism, and hence $G/\Lt =G/\Kt \times G/\Kts$. 

We choose one dimensional representations 
$\alpha_1, \alpha_2, \ldots, \alpha_s$
with kernels $K(\alpha_i)=\ker(\alpha_i)$ so that 
$$G/\Kt =G/K(\alpha_1)\times G/K(\alpha_2)\times \cdots \times 
G/K(\alpha_s)$$
Similarly, we choose  $\beta_1, \beta_2, \ldots, \beta_t$
so that 
$$\Kt/\Lt \cong G/\Kts =G/K(\beta_1)\times G/K(\beta_2)\times \cdots \times 
G/K(\beta_t). $$
This allows us to have explicit models for universal spaces: 
$$EG/\Kt=S(\infty \alpha_1)\times S(\infty \alpha_2)\times \cdots \times 
S(\infty \alpha_s) \mbox{ and } 
EG/\Kts=S(\infty \beta_1)\times S(\infty \beta_2)\times \cdots \times 
S(\infty \beta_s). $$
As filtration, we take 
$$EG/\Kt^{(n)}=S(n \alpha_1)\times S(n \alpha_2)\times \cdots \times 
S(n \alpha_s). $$
Since this is an orientable manifold, it will be equivalent to its
dual up to a shift.  To codify this, take $x_i=e(\alpha_i), y_j=e(\beta_j)$, so that 
$$H^*(BG/\Lt)\cong H^*(BG/\Kt)\tensor H^*(B\Kt /\Lt)\cong k[x_1, \ldots, x_s, y_1,
\ldots, y_s]$$
We also write 
$$\fm_{G/\Kt}=(x_1, \ldots, x_s), \fm_{G/\Kt}^{[n]}=(x_1^n, \ldots, x_s^n), $$
and note
$$H^*(BG/\Kt^{(n)})=H^*(BG/\Kt)/\fm_{G/\Kt}^{[n]}
\mbox{ and } H^*(BG/\Kt^{(n)})=\ann
(\fm_{G/\Kt}^{[n]}, H_*(BG/\Kt)). $$
It is obvious that $H^*(BG/\Kt^{(n)})$ is cyclic, but we note that it is
also self-dual. This then gives the important fact that 
$H_*(BG/\Kt^{(n)})$ is also cyclic, and generated by the
element dual to $(x_1x_2\cdots x_s)^n$ if we use the monomial
basis. We note that this is the Euler class of 
$(\alpha_1\oplus \cdots \oplus \alpha_s)^{\oplus n}$. One might say
$$H_*(BG/\Kt^{(n)})=\frac{1}{(x_1x_2\cdots x_s)^n} \cdot H^*(BG/\Kt^{(n)}).$$

\begin{remark}
This corresponds to a Spanier-Whitehead duality statement. Since
 $D(S(V)_+)\simeq \Sigma^{1-V}S(V)_+$ we 
see that
$$D(EG/\Kt^{(n)}_+)\simeq \Sigma^{s-(\alpha_1\oplus \cdots \oplus
  \alpha_s)^{\oplus n}}EG/\Kt^{(n)}_+ . $$
\end{remark}

Finally, we note that any representation $V$ of $\Lt$ with $V^{\Kt}=0$
can be written as a sum of monomials in the $\alpha_i$ and $\beta_j$
where each monomial involves at least one  $\beta_j$. Accordingly, 
$$\cE_{G/\Kt}=\{ \sum_i \lambda_i x_i +\sum_j\mu_j y_j \st (\mu_1,
\ldots , \mu_t)\neq (0, \ldots , 0)\} . $$

 \begin{remark}
 Since our construction has involved a lot of choice, it is worth 
 considering what is intrinsic once $\Kt\supseteq \Lt$ is chosen. 

 First of all,  $H^*(BG/\Kt)=k[x_1, \ldots , x_s]$ is an intrinsic 
 subalgebra of $H^*(BG/\Lt)$. From the geometric point of view, writing 
 $TG:= \spec (H^*(BG))$ we are considering the projection $TG/\Lt\lra 
 TG/\Kt$. Similarly  the ideal $\fm_{G/\Kt}$ is intrinsic, and $T\Kt/\Lt$
 is its zero set, giving the fibration 
 $$T\Kt/\Lt \lra TG/\Lt\lra TG/\Kt. $$
 Finally,  the multiplicatively closed subset $\cE_{K/L}$ is generated 
 by elements of $H^2(BG/\Lt)$ not in $\fm_{G/\Kt}$ and is also intrinsic: 
 inverting $\cE_{K/L}$ is localization at $TK/L=T\Kt/\Lt$. 

 Accordingly, the $H^*(BG/\Lt)$-module 
 $\cEi_{K/L}H^s_{\fm_{G/\Kt}}(H^*(BG/\Lt))$ (the localized $T\Kt/\Lt$-local 
 cohomology) is intrinsic. 
 \end{remark}

\begin{prop}
\label{prop:indcorep}
Evaluation at $G/\Kt$ is ind-corepresented by objects $B_{\Kt}(V,n)$ as $V$
varies through representations with $V^K=0$ and $n\geq 0$. Here 
$B_{\Kt}(V,n)(G/\Lt)=0$ unless $\Lt$ is cotoral in $\Kt$ and in that
case  we have 
$$B_{\Kt}(V,n) (G/\Lt)= S^{-V^{\Lt}}\sm H^*(BG/\Kt^{(n)})\tensor H_*(B\Kt/\Lt)$$
as $V$ varies through representations with $V^K=0$ and  $n\geq 0$.
\end{prop}

\begin{proof}
We suppose $X$ is an object of $\cAt(G)$ with torsion module $V
(G/\Kt)$ at $G/\Kt$. To define a map 
$$e:\colim_{V^K=0,n} \Hom_{\cAt(G)} (B_{\Kt}(V,n) ,  X) \lra V(G/\Kt)$$
we pick a representative $\theta  : B_{\Kt}(V,n) \lra  X$ of an
element of the domain. First
consider the  $G/\Kt$ level, and note that since $V^{\Kt}=0$, there is
no dependence on $V$.  We  take 
$e([\theta ])=\theta (G/\Kt)(\iota)$ where $\iota\in
 H^0(BG/\Kt)/\fm_{\Kt}^{[n]}$ represents  the unit. Since the maps in the
inverse sytem as $n$ varies are the identity in this
degree, this does not depend on
the choice of representative. Accordingly, the map $e$ is well
defined. 

To see $e$ is injective we note that the structure maps of $B_{\Kt}(V,n)$ are
surjective. First of all, if  $V=0$ the maps 
$B_{\Kt}(0, n)(G/\Kt) \lra B_{\Kt}(0, n)(G/\Lt)$ are surjective. 
We may as well suppose $\Kt/\Lt$ is a circle, since the general case
is a composite of such cases. Here it is easy to check that (writing $z$
for the natural representation of $\Kt/\Lt$), 
$$S^{\infty z}\sm DE\Kt/\Lt_+ \lra \Sigma S(\infty z)_+ \sm 
DE\Kt/\Lt_+ \simeq \Sigma S(\infty z)_+$$
is surjective and the general case is obtained by tensoring up. 
Now  as $V$ varies we use the fact that $e(V)$ is an isomorphism in the
domain and (since $V^K=0$) it is a monomorphism in the codomain.  

Finally, because $V (G/\Kt)$ is torsion, for any element $x$ of degree
$i$ in $V (G/\Kt)$ we can find an $n$ so that $(x_1x_2\ldots x_s)^nx=0$. 
There is therefore a map 
$$\Sigma^iH^*(BG/\Kt^{(n)})\cong \Sigma^iH^*(BG/\Kt)/\fm_{\Kt}^{[n]}
\lra V(G/\Kt)$$
with $\iota$ mapping to $x$. Next, if $\Kt$ is of codimension $c$, by
definition of $\cAt(G)$, the element $1\tensor x$ at level $G/\Kt$ only has non-zero
image at $G/L$ for finitely many connected codimension $c+1$ subgroups
$L$. For each of these we find $V_L$ so that $q_L(1\tensor x)$
is annihilated by $e(V_L)$, and take $V=\bigoplus_L V_L$. This shows
that $1\tensor x$ is in the image of a map $B_{\Kt}(V,n)\lra X$, and
so $e$ is surjective. 
\end{proof}

Having ind-corepresented evaluation at $G/\Kt$ we should describe how
the structure maps are ind-corepresented. Indeed,  if $\Lt$ is cotoral
in $\Kt$, we have maps
$$
\xymatrix{
\cEi_{\Kt/\Lt}H^*(BG/\Lt)\tensor_{H^*(BG/\Kt)}V(G/\Kt) \ar@{=}[d]\rto&
 V(G/\Lt)\ar@{=}[d]\\
\colim_{V^K=0,m}\cEi_{\Kt/\Lt}H^*(BG/\Lt)\tensor_{H^*(BG/\Kt)}\Hom(B_{\Kt}(V,m), X) \rto&
 \colim_{W^L=0,n}\Hom (B_{\Lt}(W,n), X)
}$$
In other words, given $x\in H^*(BG/\Lt)$, a representation $U$ of
$G/\Lt$ with $U^K=0$ and a representative map $\theta : B_{\Kt}(V,
m)\lra X$ we need a representative of the colimit in the
codomain. Since all maps are linear in $H^*(BG/\Lt)$, it suffices to
treat the case $x=1$. 

\begin{lemma}
Noting that $W^L=0$ implies, $W^K=0$, 
The structure map is ind-corepresented by precomposing $\theta$ with 
$$B_{\Lt}(W, n)\lra B_{\Kt}(W,n). $$
Mutliplication by $e(U)$ is given by $B_{\Kt}(V, m)\lra
B_{\Kt}(U\oplus V, m)$, so division by $e(U)$ is effected by shifting
filtration by $U$.
\end{lemma}

\begin{proof}
We only need to give the answer at $G/\Mt$ where $\Mt$ is cotoral in
$\Lt$ since otherwise the codomain is zero. At this level we are
looking at
$$S^{-V^{\Mt}}\sm H^*(BG/\Lt^{(n)})\tensor
H_*(B\Lt/\Mt)\lra \Sigma^2
S^{-V^{\Mt}}\sm H^*(BG/\Kt^{(n)})\tensor H_*(B\Kt/\Mt).  $$
We may suppose $\Kt/\Lt$ is a circle. So the only effect is that
$H^*(B\Kt/\Lt)$ is removed from the cohomology in the domain and replaced by
$\Sigma^2 H_*(B\Kt/\Lt)$ in the homology of the codomain. As in the
proof of Proposition \ref{prop:indcorep}, writing $z$
for the natural representation of $\Kt/\Lt$ we note that it is 
$$S^{\infty z}\sm DE\Kt/\Lt_+ \lra \Sigma S(\infty z)_+ \sm 
DE\Kt/\Lt_+ \simeq \Sigma S(\infty z)_+$$
\end{proof}

\begin{remark}
This gives another approach to constructing injectives. 
Indeed, by definition, if $I$ is an injective $H^*(BG/\Lt)$-module,
then  $\at_L(I)$ has the property 
$$\Hom_{\cAt(G)}(X, \at_L (I))=\Hom_{H^*(BG/\Lt)}(V(G/\Lt), I)$$
so if $\Lt$ is cotoral in $\Kt$, we must have
\begin{multline*}
\at_L(I)(G/\Kt)=
\colim_{V^K=0, n}\Hom_{H^*(BG/\Lt)}(B_{\Kt}(V, n)(G/\Lt), I)\\
=\colim_{V^K=0, n}\Hom_{H^*(BG/\Lt)}(S^{-V^L}\tensor 
H^*(BG/\Kt^{(n)})\tensor H_*(B \Kt/\Lt), I). 
\end{multline*}
We now take the generating torsion injective $I=H_*(BG/\Lt)$ so that 
$\Hom_{H^*(BG/\Lt)}(\cdot , I)=\Hom_k (\cdot ,k)$.  As noted above 
$$H^*(BG/\Kt^{(n)})\tensor H_*(B \Kt/\Lt) 
\cong \Sigma^{s-na}H_*(BG/\Kt^{(n)})\tensor H_*(B \Kt/\Lt) $$
where $a=|x_1x_2\cdots x_s|$. Accordingly the value is 
$$S^{+V^L}\tensor \Sigma^{na-s}
H^*(BG/\Kt^{(n)}\times B\Kt/\Lt). $$
The colimit is 
$$\cEi_{K/L}H_*(BG/\Kt)\tensor
H^*(B\Kt/\Lt)=\cEi_{K/L}H^s_{\fm_{\Kt}}(H^*(BG/\Lt)).  $$
This agrees with the previous construction by Lemma
\ref{lem:injasloccoh}. 
Indeed, it was the  calculation by ind-corepresentability that 
first alerted the author to the connection between injectives and 
local cohomology, and hence led to the formulation of Lemma 
\ref{lem:injasloccoh}. 
\end{remark}

\section{The Adams spectral sequence}
\label{sec:AtASS}
The paper has been leading up to the construction of a method of
calculation based on the abelian torsion model. 

\begin{thm}
\label{thm:AtASS}
There is an Adams spectral sequence 
$$\ExtAt^{*,*} (\piAts (X), \piAts (Y))\Rightarrow [X,Y]^G_*. $$
This is a finite, strongly convergent spectral sequence. 
\end{thm}

There is a standard method for constructing an Adams spectral
sequence: we first establish that enough
injectives are realizable and that the Adams spectral sequence applies
to them, and that the homology theory detects triviality.

\begin{lemma}
\label{lem:rinj}
For any injective $\cOcFL$-module $I$, there is a $G$-spectrum
$A_L(I)$ realizing $\at_L(I)$ in the sense that
$\piAts(A_L(I))=\at_L(I)$ and $\piAts$ gives an isomorphism
$$\piAts: [X, A_L(I)]^G_*\stackrel{\cong}\lra \HomAt(\piAts (X), \at_L(I)). $$
\end{lemma}

\begin{proof}
For an injective, torsion $\cOcFL$-module $I$, the functor 
$$X\longmapsto A_L(I)_G^*(X) =\HomAt(\piAts (X), \at_L(I)) 
=\Hom_{\cOcFH}(V_X(\cF /L), I)$$
is exact because $I$ is injective and is therefore a cohomology 
theory on $G$-spectra, so by Brown Representability there is a $G$-spectrum $A_L(I)$ so that 
$$[X, A_L(I)]^G_*=\HomAt(\piAts (X), \at_L(I)). $$
The isomorphism  $\piAts (A_L(I))=\at_L(I)$ follows from Proposition \ref{prop:indcorep}.
\end{proof}

\begin{lemma}
\label{lem:piAtsfaithful}
If $\piAts(X)=0$ then $X\simeq *$. 
\end{lemma}

\begin{proof}
If $\piAts (X)=0$ then $H^*_{G/\Kt}(\Phi^{\Kt}X)=0$ for all subgroups
$\Kt$. The geometric isotropy of $X$ is therefore empty and $X$ is
contractible by the Geometric Fixed Point Whitehead Theorem.
\end{proof}

\begin{proof}[of Theorem \ref{thm:AtASS}]
As usual we need only show that enough injectives are realizable,
that  the spectral sequence is correct for maps into these spectra and
that the spectral sequence is convergent. 

In more detail, we take an injective resolution of $\piAts (Y)$: 
$$0\lra \piAts(Y)\lra I_0\lra I_1\lra \cdots \lra I_n\lra 0. $$
This is finite by Proposition \ref{prop:id}, and by Lemma \ref{lem:enoughinj} we may assume 
each $I_s$ is a sum of injectives $\at_K(I)$ for subgroups $K$
and injective $\cOcFK$-modules $I=H_*(BG/\Kt)$ where $\Kt$ has
identity component $K$.

By Lemma \ref{lem:rinj} this is realizable by a tower
$$\xymatrix{
Y\ar@{=}[r]&Y_0\dto &Y_1\lto\dto &Y_2\lto\dto &\cdots\lto  & Y_n\lto\dto &Y_{n+1}\lto\\
                     &\bbI_0&\Sigma^{-1}\bbI_1&\Sigma^{-2}\bbI_2&&\Sigma^{-n}\bbI_n&
}$$
which is built inductively, starting with $Y\lra \bbI_0$ realizing
$\piAts(Y)\lra I_0$ and taking $Y_1$ to be the fibre. Once $Y_s$ has
been defined as the fibre of $Y_{s-1}\lra \Sigma^{-s+1}\bbI$, we see
that $\piAts(Y_s)=\Sigma^{-s}(\im (I_{s-1}\lra
I_s))$. In particular, $\piAts(Y_{n+1})=0$ and $Y_{n+1}\simeq * $ by Lemma
\ref{lem:piAtsfaithful}.

We obtain the spectral sequence by applying $[X, \cdot]^G_*$ to the
tower. By Lemma \ref{lem:rinj} the $E_1$ term is $\HomAt(\piAts
(X), I_\bullet)$ and therefore the $E_2$-term is as stated. Strong
convergence is clear because the filtration is finite. 
\end{proof}


\begin{thebibliography}{1}
\bibitem{Torsion1}
S.~Balchin, J.P.C. Greenlees, L.~Pol, and J.~Williamson.
\newblock ``Torsion models: the one step case.''
\newblock {\em AGT (to appear), 35pp}, arXiv: 2011.10413
\bibitem{s1q}
J.~P.~C. Greenlees.
\newblock ``Rational {$S\sp 1$}-equivariant stable homotopy theory.''
\newblock {\em Mem. Amer. Math. Soc.}, 138(661):xii+289, 1999.
\bibitem{tnq1} J.P.C.Greenlees 
      ``Rational torus-equivariant stable homotopy I:
    calculating groups of stable maps.''
    JPAA {\bf 212} (2008) 72-98
    (http://dx.doi.org/10.1016/j.jpaa.2007.05.010), arXiv:0705.2686
\bibitem{tnq2} J.P.C.Greenlees 
      ``Rational torus-equivariant stable homotopy II:
 the algebra of localization and inflation.''
 JPAA {\bf 216} (2012) 2141-2158, arXiv:1108.4868
\bibitem{tnq3} J.P.C.Greenlees 
``Rational torus-equivariant stable homotopy III: comparison of
models.''
JPAA {\bf 220} (2016) 3573-3609, arXiv:1410.5464
\bibitem{tnq4}
J.P.C.Greenlees 
``Rational torus-equivariant stable homotopy IV: thick subcategories and the Balmer spectrum for finite 
spectra.'' 
Preprint, 26pp arXiv:1612.01741
\bibitem{spcgq} J.P.C.Greenlees 
``The Balmer spectrum for rational equivariant cohomology theories'' 
JPAA {\bf 223} (2019) 2845-2871 arXiv: 1706.07868 
\bibitem{gfreeq}
J.~P.~C. Greenlees and B.~Shipley.
\newblock An algebraic model for free rational {$G$}-spectra for connected
  compact {L}ie groups {$G$}.
\newblock {\em Math. Z.}, 269(1-2):373--400, 2011.

\end{thebibliography}
\end{document}